\documentclass[a4paper]{amsart}

\usepackage{amsfonts}
\usepackage{amssymb}
\usepackage{amstext}
\usepackage{graphicx}
\usepackage{tikz}
\usepackage{caption}

\usepackage[]{amscd}
\usepackage{hyperref}
\newtheorem{theorem}{Theorem}[section]
\newtheorem{lemma}{Lemma}[section]
\newtheorem{corollary}{Corollary}[section]
\newtheorem{definition}{Definition}[section]
\newtheorem{remark}{Remark}[section]
\newtheorem{proposition}{Proposition}[section]
\title{Ap\'ery-Fermi pencil of $K3$-surfaces and their $2$-isogenies}

\begin{document}

\author[M.-J. BERTIN and Odile LECACHEUX]{Marie Jos\'e BERTIN and Odile Lecacheux}


\date{\today}

\curraddr{Sorbonne Universit\'e\\ Institut
de Math\'ematiques de Jussieu-Paris Rive Gauche\\
Case 247\\
 4 Place Jussieu, 75252 PARIS Cedex 05, France}

\email{marie-jose.bertin@imj-prg.fr\\
odile.lecacheux@imj-prg.fr}

\begin{abstract}
Given a generic $K3$ surface $Y_k$ of the Ap\'ery-Fermi pencil, we use 
the Kneser-Nishiyama technique to determine all its non isomorphic elliptic fibrations.
These computations lead to determine those fibrations with 2-torsion sections T. We classify the fibrations such that the translation by T gives a Shioda-Inose structure. The other fibrations correspond to a K3 surface identified by it transcendental lattice.
The same problem is solved for a singular 
member $Y_2$ of the family showing the differences with the generic 
case. In conclusion we put our results in the context of relations 
between $2$-isogenies and isometries on the singular surfaces of the family.
\end{abstract}

\maketitle
%\tableofcontents.
\section{Introduction}

The Ap\'ery-Fermi pencil $\mathcal{F}$ is realised with the equations
\[X+\frac{1}{X}+Y+\frac{1}{Y}+Z+\frac{1}{Z}=k, \,\,\,\,\,\,k\in \mathbb Z,\]
and taking $k=s+\frac{1}{s}$, is seen as the Fermi threefold $\mathcal Z$ with compactification denoted $\bar{\mathcal Z}$ \cite{PS}.

The projection $\pi_s :\bar{\mathcal Z}\rightarrow \mathbb P^1(s)$ is called the Fermi fibration. In their paper \cite{PS}, Peters and Stienstra proved that for $s\notin \{0,\infty,\pm 1, 3\pm 2\sqrt{2},-3\pm 2\sqrt{2} \}$ the fibers of the Fermi fibration are $K3$ surfaces with the N\'eron-Severi lattice of the generic fiber isometric to $E_8(-1)\oplus E_8(-1) \oplus U \oplus \langle -12 \rangle $ and transcendental lattice isometric to $T=U\oplus \langle 12 \rangle $ ($U$ denotes the hyperbolic lattice and $E_8$ the unimodular lattice of rank $8$).
Hence this family appears as a family of $M_6$-polarized $K3$-surfaces $Y_k$ with period $t\in \mathcal H$. And we deduce from a result of Dolgachev \cite{D} the following property. Let $E_t=\mathbb C/\mathbb Z+t\mathbb Z$ and $E'_t=\mathbb C/\mathbb Z+(-\frac{1}{6t})\mathbb Z$ be the corresponding pair of isogenous elliptic curves. Then there exists a canonical involution $\tau$ on $Y_k$ such that $Y_k/(\tau)$ is birationally isomorphic to the Kummer surface $E_t \times E'_t/(\pm 1)$.

This result is linked to the Shioda-Inose structure of $K3$-surfaces with Picard number $19$ and $20$ described first by Shioda and Inose \cite{SI} and extended by Morrison \cite{M}.

As observed by Elkies \cite{E}, the base of the pencil of $K3$ surfaces can be identified with the elliptic modular curve $X_0(6)/\langle w_2,w_3 \rangle $. Indeed it can be derived from Peters and Stienstra \cite{PS}.

%Let us recall that the Shioda's Sandwich theorem \cite{Shio}, explicites the Shioda-Inose structure for elliptic $K3$-surfaces (with a section) with two $II^*$-singular fibers, illustrating the fact that such a $K3$-surface is ``isogenous'' to a Kummer surface $K$ in a very concrete sense.

% This family has also been studied from other points of view by Verrill \cite{V}, then Bertin \cite{Be}. Moreover the periods of the Ap\'ery-Fermi pencil satisfy a Picard-Fuchs equation given by Peters and Stienstra \cite{PS} and used by Bertin \cite{Be} to express the Mahler measure of the Laurent polynomials of the family in terms of the $L$-series of the corresponding $K3$-surface, denoted $Y_k$, when this latter has Picard number $20$. Bertin observed also that $Y_2$ and $Y_{10}$ have the same $L$-series and suspected an algebraic relation between $Y_2$ and $Y_{10}$. 
In \cite{Shio}, Shioda considers the  problem whether every Shioda-Inose structure can be
extended to a sandwich. More precisely Shioda proved a ''Kummer sandwich
theorem'' that is, for an elliptic $K3$-surface $X$ (with a section) with two
II*-fibres, there exists a unique Kummer surface $S=Km(C_{1}\times C_{2})$
with two rational maps of degree 2, $X\rightarrow S$ and $S\rightarrow X$ where
$C_{1}$ and $C_{2}$ are elliptic curves.

In van Geemen-Sarti \cite{G}, Comparin-Garbagnati \cite{C}, Koike \cite{Ko} and  Sch\"{u}tt \cite{Sc}, sandwich Shioda-Inose
structures are constructed via elliptic fibrations with 2-torsion sections. 

  Recently Bertin and Lecacheux \cite{BL} found all the elliptic fibrations of a singular member $Y_2$ of $\mathcal{F}$ (i.e. of Picard number $20$) and observed that many of its elliptic fibrations are endowed with $2$-torsion sections. Thus a question arises: are the corresponding $2$-isogenies between $Y_2$ and this new $K3$-surface $S_2$ all Morrison-Nikulin meaning that $S_2$ is Kummer? Observing also that the Shioda's Kummer sandwiching between a $K3$ surface $S$ and its Kummer $K$ is in fact a $2$-isogeny between two elliptic fibrations of $S$ and $K$, we extended the above question to the generic member $Y_k$ of the family $\mathcal{F}$ and obtained the following results.

\begin{theorem}
Suppose $Y_k$ is a generic $K3$ surface of the family with Picard number $19$.

Let $\pi:Y_{k}\rightarrow\mathbb{P}^{1}$ be an elliptic fibration with a torsion section of order $2$ which defines an involution $i$ of $Y_{k}$ (Van Geemen-Sarti involution) then the quotient $Y_{k}/i$ is either the Kummer surface $K_{k}$ associated to $Y_{k}$ given by its Shioda-Inose structure or a surface $S_k$ with transcendental lattice  $T_{S_k}= \langle -2 \rangle \oplus  \langle 2 \rangle  \oplus  \langle 6 \rangle $ and N\'eron-Severi lattice $NS(S_k)=U\oplus E_8(-1)\oplus E_7(-1) \oplus \langle (-2) \rangle \oplus \langle (-6) \rangle $, which is not a Kummer surface by a result of Morrison \cite{M}. The $K3$ surface $S_k$ is the Hessian $K3$ surface of a general cubic surface with $3$ nodes studied by Dardanelli and van Geemen \cite{DG}. Thus, $\pi$ leads to an elliptic fibration either of $K_k$ or of $S_k$. Moreover there exist some genus $1$ fibrations  $\theta:K_{k}$ $\rightarrow \mathbb{P}^{1}$ without section such that their Jacobian variety satisfies $J_{\theta}\left(  K_{k}\right)=S_k$. 

More precisely, among the elliptic fibrations of $Y_{k}$ (up to automorphisms) $12$ of them have a
two-torsion section.  And only $7$ of them possess a Morisson-Nikulin involution $i$ such that $Y_k/i=K_{k}$.
\end{theorem}

%For $k=2,$ let $\pi...$ (m\^{e}me propri\'{e}t\'{e}s que dans le cas g\'{e}n\'{e}ral) or $Y_{2}/i\equiv Y_{2}.$ In the second case $\pi$ defines the same class of fibration (up to an automorphism of the base $\mathbb{P}$ or exchange two different class of fibrations of $Y_{2}.$

\begin{theorem}

In the Ap\'ery-Fermi pencil, the $K3$-surface $Y_2$ is singular, meaning that its Picard number is $20$. Moreover $Y_2$ has many more $2$-torsion sections than the generic $K3$ surface $Y_k$; hence among its $19$ Van Geemen-Sarti involutions, $13$ of them are Morrison-Nikulin involutions, $5$ are symplectic automorphisms of order $2$ (self-involutions) and one exchanges two elliptic fibrations of $Y_2$.

The specializations to $Y_2$ of the $7$ Morrison-Nikulin involutions of a generic member $Y_k$ are verified among the $13$ Morrison-Nikulin involutions of $Y_2$, as proved in a general setting by Sch\"{u}tt \cite{Sc}. The specializations of the $5$ involutions between $Y_k$ and the $K3$-surface $S_k$ are among the $6$ Van Geemen-Sarti involutions of $Y_2$ which are not Morrison-Nikulin.

\end{theorem}

This theorem provides an example of a Kummer surface $K_2$ defined by the product of two isogenous elliptic curves (actually the same elliptic curve of $j$-invariant equal to $8000$), having many fibrations of genus one  whose Jacobian surface is not a Kummer surface. A similar result but concerning a Kummer surface defined by two non-isogenous elliptic curves has been exhibited by Keum \cite{K}.

Throughout the paper we use the following result \cite{Si}. If $E$ denotes an elliptic fibration with a 2-torsion point $(0,0)$:
\[ E: y^2=x^3+Ax^2+Bx,\]
the quotient curve $E/\langle(0,0)\rangle $ has a Weierstrass equation of the form
\[E/\langle (0,0) \rangle : y^2=x^3-2Ax^2+(A^2-4B)x.\]

The paper is organised as follows.

In section $2$ we recall the Kneser-Nishiyama method and use it to find all the $27$ elliptic fibrations of a generic $K3$ of the family $\mathcal{F}$. In section $3$, using Elkies's method of "$2$-neighbors" \cite{El}, we exhibit an elliptic parameter giving a Weierstrass equation of the elliptic fibration. The results are summarized in Table 2. Thus we obtain all the Weierstrass equations of the $12$ elliptic fibrations with $2$-torsion sections. Their $2$-isogenous elliptic fibrations are computed in section $5$ with their Mordell-Weil groups and discriminants. 
Section $4$  recalls generalities about Nikulin involutions and  Shioda-Inose structure. Section $5$ is devoted to the proof of Theorem 1.1 while section $6$ is concerned with the proof of Theorem 1.2.

In the last section $7$, using a theorem of Boissi\` ere, Sarti and Veniani \cite{BSV}, we explain why Theorem 1.2 cannot be generalised to the other singular K3 surfaces of the family.

Computations were performed using partly the computer algebra system PARI \cite{PA} and mostly the computer algebra system MAPLE and the Maple Library ``Elliptic Surface Calculator'' written by Kuwata \cite{Ku1}.

\section{Elliptic fibrations of the family}

We refer to \cite{BL}, \cite{Sc-Shio} for definitions concerning lattices, primitive embeddings, orthogonal complement of a sublattice into a lattice. We recall only what is essential for understanding this section and section 5.2.
\subsection{Discriminant forms}

Let $L$ be a non-degenerate lattice. 
The \textbf{dual lattice} $L^*$ of $L$ is defined by
$$L^*:=\text{Hom}(L,\mathbb Z) =\{x\in L \otimes \mathbb Q /\,\,\, b(x,y)\in \mathbb Z \hbox{  for all }y \in L \}$$
%Obviously $L^*$ is an overlattice of $L$ but its form takes non-integer values too. The canonical bilinear form on $L^*$ induced by $b$ is denoted by the same letter.
and the \textbf{discriminant group $G_L$} by
$$G_L:=L^*/L.$$
This group is finite if and only if $L$ is non-degenerate. In the latter case, its order is equal to the absolute value of the lattice determinant $\mid \det (G(e)) \mid$ for any basis $e$ of $L$.
A lattice $L$ is \textbf{unimodular} if $G_L$ is trivial.

Let $G_L$ be the discriminant group of a non-degenerate lattice $L$. The bilinear form on $L$ extends naturally to a $\mathbb Q$-valued symmetric bilinear form on $L^*$ and induces a symmetric bilinear form 
$$b_L: G_L \times G_L \rightarrow\mathbb Q / \mathbb Z.$$
If $L$ is even, then $b_L$ is the symmetric bilinear form associated to the quadratic form defined by
$$
\begin{matrix}
q_L: G_L  &\rightarrow  & \mathbb Q/2\mathbb Z\\
q_L(x+L) & \mapsto & x^2+2\mathbb Z.
\end{matrix}
$$

The latter means that $q_L(na)=n^2q_L(a)$ for all $n\in \mathbb Z$, $a\in G_L$ and $b_L(a,a')=\frac {1}{2}(q_L(a+a')-q_L(a)-q_L(a'))$, for all $a,a' \in G_L$, where $\frac {1}{2}:\mathbb Q/2 \mathbb Z \rightarrow \mathbb Q / \mathbb Z$ is the natural isomorphism.
The pair $\boldsymbol{(G_L,b_L)}$ (resp. $\boldsymbol{(G_L,q_L)}$) is called the \textbf{discriminant bilinear} (resp. \textbf{quadratic}) \textbf{form} of $L$.

The \textbf{ lattices} $A_n=\langle a_1,a_2,\ldots ,a_n \rangle$ ($n\geq 1$), $D_l=\langle d_1,d_2,\ldots ,d_l \rangle$ ($l\geq 4$), $E_p=\langle e_1,e_2,\ldots ,e_p \rangle$ ($p=6,7,8$) defined by the following \textbf{Dynkin diagrams} are called the \textbf{root lattices}. All the vertices $a_j$, $d_k$, $e_l$ are roots and two vertices $a_j$ and $a_j'$ are joined by a line if and only if $b(a_j,a_j')=1$. We use Bourbaki's definitions \cite{Bou}. The discriminant groups of these root lattices are given below.

\noindent
\begin{minipage}[t]{0.58\textwidth}
$\mathbf{A_n,\,\, G_{A_n}}$

Set 

$[1]_n=\frac {1}{n+1} \sum_{j=1}^{n}(n-j+1)a_j$

then $A_n^*=\langle A_n,[1]_n \rangle$ and 

$ G_{A_n}=A_n^*/A_n  \simeq \mathbb Z /(n+1)\mathbb Z .$

$q_{A_n}([1]_n)=-\frac{n}{n+1}.$
\end{minipage}
\begin{minipage}[t]{0.4\textwidth}
\bigskip
\begin{center}
\begin{tikzpicture}[scale=0.4]%
\draw[fill=black](0.1,1.7039)circle (0.0865cm)node [below] {$a_1$};
\draw[fill=black](2.9,1.7039)circle (0.0865cm)node [below] {$a_2$};
\draw[fill=black](5.7,1.7039)circle (0.0865cm)node [below] {$a_3$};
\draw[fill=black](9.9,1.7039)circle (0.0865cm)node [below]{$a_n$};
\draw[thick](0.1,1.7039)--(6.81,1.7039);
\draw[dashed,thick](7.3,1.7039)--(8.6,1.7039);
\draw[thick](9.01,1.7039)--(9.9,1.7039);
\end{tikzpicture}
\end{center}

\end{minipage}

\bigskip
\noindent
\begin{minipage}[t]{.58 \textwidth}
$\mathbf{D_l,\,\, G_{D_l}}.$ 

Set

\noindent
$[1]_{D_l}  = \frac {1}{2}\left( \sum_{i=1}^{l-2}id_i+\frac {1}{2}(l-2)d_{l-1}+\frac {1}{2}ld_l \right)$

\noindent
$[2]_{D_l}  = \sum_{i=1}^{l-2} d_i+\frac {1}{2}(d_{l-1}+d_l)$

\noindent
$[3]_{D_l}  = \frac {1}{2}\left( \sum_{i=1}^{l-2}id_i+\frac {1}{2}ld_{l-1}+\frac {1}{2}(l-2)d_l  \right)$

\noindent
then $D_l^*= \langle D_l, [1]_{D_l}, [3]_{D_l} \rangle,$

\noindent
$G_{D_l}=D_l^*/D_l= \langle [1]_{D_l} \rangle \simeq \mathbb Z/4\mathbb Z \,\,\,\,\text{if } l \text{ is odd,}$

\end{minipage}
\begin{minipage}[t]{.4 \textwidth}
\smallskip
\begin{center}
\begin{tikzpicture}[scale=0.4]%
\draw[fill=black](0.1,1.7039)circle (0.0865cm)node [below] {$d_l$};
\draw[fill=black](2.9,1.7039)circle (0.0865cm)node [below] {$d_{l-2}$};
\draw[fill=black](5.7,1.7039)circle (0.0865cm)node [below] {$d_{l-3}$};
\draw[fill=black](9.9,1.7039)circle (0.0865cm)node [below]{$d_1$};
\draw[fill=black](2.9,3.10)circle (0.0865cm)node [above]{$d_{l-1}$};
\draw[thick](0.1,1.7039)--(6.81,1.7039);
\draw[dashed,thick](7.3,1.7039)--(8.6,1.7039);
\draw[thick](9.01,1.7039)--(9.9,1.7039);
\draw[thick](2.9,1.7039)--(2.9,3.10);
\end{tikzpicture}

\end{center}
\end{minipage}

\noindent
$G_{D_l}=D_l^*/D_l= \langle [1]_{D_l}, [2]_{D_l} \rangle \simeq \mathbb Z/2\mathbb Z \times \mathbb Z/2\mathbb Z \,\,\,\,\text{if }l \text{ is even}.$

\noindent
$q_{D_l}([1]_{D_l})=-\frac{l}{4},\, \,\,q_{D_l}([2]_{D_l})=-1,\,\,\,b_{D_l}([1],[2])=-\frac{1}{2}. $

\bigskip
\noindent
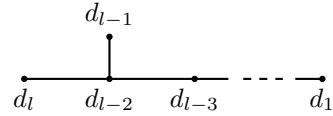
\begin{minipage}[t]{0.58\textwidth}
$\mathbf{E_p,\,\,G_{E_p}}\,\quad p=6,7,8.$

Set

$[1]_{E_6}:=\eta_6 =-\frac {1}{3}(2e_1+3e_2+4e_3+6e_4+5e_5+4e_6)$  and

$[1]_{E_7}:=\eta_7 =-\frac {1}{2}(2e_1+3e_2+4e_3+6e_4+5e_5+4e_6+3e_7)$,

then
$E_6^*=\langle E_6, \eta_6 \rangle $, $E_7^*=\langle E_7, \eta_7 \rangle$ and 
$E_8^*=E_8.$
\end{minipage}
\begin{minipage}[t]{0.4\textwidth}
\smallskip
\begin{center}
 \begin{tikzpicture}[scale=0.4]%
\draw[fill=black](0.1,1.7039)circle (0.0765cm)node [below] {$e_1$};
\,\,\draw[fill=black](2.9,1.7039)circle (0.0765cm)node [below] {$e_{3}$};
\draw[fill=black](5.7,1.7039)circle (0.0765cm)node [below] {$e_{4}$};
\draw[fill=black](9.9,1.7039)circle (0.0765cm)node [below]{$e_p$};
\draw[fill=black](5.7,3.10)circle (0.0765cm)node [above]{$e_{2}$};
\draw[thick](0.1,1.7039)--(6.81,1.7039);
\draw[dashed,thick](7.3,1.7039)--(8.6,1.7039);
\draw[thick](9.01,1.7039)--(9.9,1.7039);
\draw[thick](5.7,1.7039)--(5.7,3.10);
\end{tikzpicture}
\end{center}

\end{minipage}

\noindent
$G_{E_6}=E_6^*/E_6\simeq \mathbb Z /3 \mathbb{ Z},\, \,\,G_{E_7}=E_7^*/E_7 \simeq \mathbb Z /2 \mathbb Z,$  

\noindent

$q_{E_6(\eta_6)}=-\frac{4}{3}\,\,\,\,q_{E_7(\eta_7)}=-\frac{3}{2}.$

Let $L$ be a Niemeier lattice (i.e. an unimodular lattice of rank $24$). Denote $L_{\text{root}}$ its root lattice. We often write $L=Ni(L_{\text{root}})$. Elements of $L$ are defined by the glue code composed with glue vectors. Take for example $L=Ni(A_{11}D_7E_6)$. Its glue code is generated by the glue vector $[1,1,1]$ where the first $1$ means $[1]_{A_{11}}$, the second $1$ means $[1]_{D_7}$ and the third $1$ means $[1]_{E_6}$. In the glue code $\langle[1,(0,1,2)]\rangle$, the notation $(0,1,2)$ means any circular permutation of $(0,1,2)$. Niemeier lattices, their root lattices and glue codes used in the paper are given in Table \ref{Ta:Ni} (glue codes are taken from Conway and Sloane \cite{Co}).

\begin{table}

\begin{center}
\begin{tabular}{|l|l|l|}
\hline
$L_{\text{root}}$ & $L/L_{\text{root}}$   & \text{glue vectors}  \\\hline
$E_8^3$&$(0)$&$0$\\\hline
$D_{16}E_8$ & $\mathbb Z/2 \mathbb Z$ &$\langle[1,0]\rangle$   \\\hline
$D_{10}E_7^2$ & $(\mathbb Z/2 \mathbb Z)^2$ & $\langle[1,1,0],[3,0,1]\rangle$\\ \hline
$A_{17}E_7$ & $\mathbb Z/6 \mathbb Z$ & $\langle[3,1]\rangle $ \\\hline
$D_{24}$ & $\mathbb Z/2 \mathbb Z$ & $\langle[1]\rangle$\\ \hline
$D_{12}^2$ & $(\mathbb Z/2 \mathbb Z)^2$ & $\langle[1,2],[2,1]\rangle$ \\ \hline
$D_8^3$ &  $(\mathbb Z/2 \mathbb Z)^3$ & $\langle[1,2,2],[1,1,1],[2,2,1]\rangle$ \\ \hline
$A_{15} D_9$ & $\mathbb Z/8 \mathbb Z$ & $\langle[2,1]\rangle$ \\ \hline
$E_6^4$ &  $(\mathbb Z/3 \mathbb Z)^2$ & $\langle[1,(0,1,2)]\rangle$ \\ \hline
$A_{11}D_7E_6$ &  $\mathbb Z/12 \mathbb Z$ & $\langle[1,1,1]\rangle$ \\ \hline
$D_6^4$ &  $(\mathbb Z/2 \mathbb Z)^4$ & $\langle\text{even permutations of } [0,1,2,3]\rangle$\\ \hline
$A_9^2D_6$ & $\mathbb Z/10 \mathbb Z \times \mathbb Z/2 \mathbb Z$ & $\langle[2,4,0],[5,0,1],[0,5,3]\rangle$ \\ \hline
$A_7^2D_5^2$ &  $\mathbb Z/8 \mathbb Z \times \mathbb Z/4 \mathbb Z$ & $\langle[1,1,1,2],[1,7,2,1]\rangle$ \\ \hline

 \end{tabular}
\end{center}
\caption{Some Niemeier lattices and their glue codes \cite{Co}}\label{Ta:Ni}
\end{table}

\subsection{Kneser-Nishiyama technique}

%Denote $Y_k$ a generic $K3$ surface of the Ap\'ery-Fermi pencil of Picard number $19$.

We use the Kneser-Nishiyama method to determine all the elliptic fibrations of $Y_k$. For further details we refer to \cite{Nis}, \cite{Sc-Shio}, \cite{BL}, \cite {BGL}. In  \cite{Nis}, \cite{BL}, \cite {BGL} only singular $K3$ (i.e. of Picard number $20$) are considered. In this paper we follow \cite{Sc-Shio} we briefly recall.

Let $T(Y_k)$ be the transcendental lattice of $Y_k$, that is the orthogonal complement of $NS(Y_k)$ in $H^2(Y_k,\mathbb Z)$ with respect to the cup-product. The lattice $T(Y_k)$ is an even lattice of rank $r=22-19=3$ and signature $(2,1)$. Since $t=r-2=1$, $T(Y_k)[-1]$ admits a primitive embedding into the following indefinite unimodular lattice:
\[T(Y_k)[-1] \hookrightarrow U \oplus E_8\] 
where $U$ denotes the hyperbolic lattice and $E_8$ the unimodular lattice of rank $8$. Define $M$ as the orthogonal complement of a primitive embedding of $T(Y_k)[-1]$ in $U\oplus E_8$. Since
\[ T(Y_k)[-1]= \begin{pmatrix}

                   0 & 0 & -1 \\
                   0 & -12 & 0 \\
                   -1 & 0 & 0

               \end{pmatrix} ,\]
it suffices to get a primitive embedding of $(-12)$ into $E_8$. From Nishiyama \cite{Nis} we find the following primitive embedding:
\[ v= \langle 9e_2+6e_1+12e_3+18e_4+15e_5+12e_6+8e_7+4e_8 \rangle \hookrightarrow E_8,\]
giving $(v)_{E_8}^{\perp}=A_2 \oplus D_5$. Now the primitive embedding of $T(Y_k)[-1]$ in $U\oplus E_8$ is defined by $U\oplus v$; hence $M=(U\oplus v)_{U\oplus E_8}^{\perp} =A_2 \oplus D_5$. By construction, this lattice is negative definite of rank $t+6=1+6=r+4=3+4=26-\rho(Y_k)=7$ with discriminant form $q_M=-q_{T(Y_k)[-1]}=q_{T(Y_k)}=-q_{NS(Y_k)}$. Hence $M$ takes exactly the shape required for Nishiyama's technique.

All the elliptic fibrations come from all the primitive embeddings of $M=A_2 \oplus D_5$ into all the Niemeier lattices 
$L$. Since $M$ is a root lattice, a primitive embedding of $M$ into $L$ is in fact a primitive embedding into $L_{\text{root}}$. Whenever the primitive embedding is given by a primitive embedding of $A_2$ and $D_5$ in two different factors of $L_{\text{root}}$, or for the primitive embedding of $M$ into $E_8$, we use Nishiyama's results \cite{Nis}. Otherwise we have to determine the primitive embeddings of $M$ into $D_l$ for $l=8,9,10,12,16,24$. This is done in the following lemma.

\begin{lemma}\label{lem:2.1}
We obtain the following primitive embeddings.
\begin{enumerate}
\item
\[A_2\oplus D_5= \langle d_8,d_6,d_7,d_5,d_4,d_1,d_2 \rangle \hookrightarrow D_8\]
\[ \langle d_8,d_6,d_7,d_5,d_4,d_1,d_2 \rangle _{D_8}^{\perp}= \langle 2d_1+4d_2+6d_3+6d_4+6d_5+6d_6+3d_7+3d_8 \rangle =(-12)\]
\item
\[A_2\oplus D_5= \langle d_9,d_7,d_8,d_6,d_5,d_3,d_2 \rangle \hookrightarrow D_9\]
\[ \langle d_9,d_7,d_8,d_6,d_5,d_3,d_2 \rangle_{D_9}^{\perp}= \langle d_9+d_8+2d_7+2d_6+2d_5+2d_4+d_3-d_1,d_3+2d_2+3d_1 \rangle \]
with Gram matrix $\begin{pmatrix}      
                       -4 & 6 \\
                         6 & -12
                        \end{pmatrix}$   
  of determinant $12$.
\item
\[A_2\oplus D_5= \langle d_{n},d_{n-2},d_{n-1},d_{n-3},d_{n-4},d_{n-7},d_{n-6} \rangle \hookrightarrow D_{n}, n\geq 10\]
\[ \langle d_{n},d_{n-2},d_{n-1},d_{n-3},d_{n-4},d_{n-7},d_{n-6} \rangle_{D_{n}}^{\perp}=\]
\[  \langle a=d_n+d_{n-1}+2(d_{n-2}+...+d_2)+d_1,d_{n-6}+2d_{n-7}+3d_{n-8},d_{n-9},...,d_1 \rangle \]
%\[(A_2 \oplus D_5 )_{\text{root}}^{\perp}=D_{n-8}.\]
\[((A_2 \oplus D_5 )_{D_n}^{\perp})_{\text{root}}=D_{n-8}.\]
We have also the relation $2.[2]_{D_n}=a+d_1$, $a$ being the above root.

\end{enumerate}
\end{lemma}

\begin{theorem}
There are $27$ elliptic fibrations on the generic $K3$ surface of the Ap\'ery-Fermi pencil (i.e. with Picard number $19$). They are derived from all the non isomorphic primitive embeddings of $A_2\oplus D_5$ into the various Niemeier lattices. Among them, $4$ fibrations have rank $0$, precisely with the type of singular fibers and torsion.

\[
\begin{matrix}
  & A_{11} 2A_2 2A_1&  6-{\text{torsion}}\\
  & E_6 D_{11} & 0-{\text{torsion}}\\
  & E_7 A_5 D_5 & 2-{\text{torsion}}\\
  & E_8 E_6 A_3 & 0-{\text{torsion}}.
\end{matrix}
\]
The list together with the rank and torsion is given in Table \ref{Ta:Fib2}.
\end{theorem}

\begin{table}[tp]\footnotesize

\begin{center}
\begin{tabular}{|c|c|c|c|c|c|c|}
\hline
$L_{\text{root}}$ & $L/L_{\text{root}}$ &  & & \text{type of Fibers} & \text{Rk} & \text{Tors.} \\ \hline
  $E_8^3$    &  $(0)$ & & & & &\\ \hline
& \#1  & $A_2\subset E_8$ & $D_5 \subset E_8$  & $E_6 A_3 E_8$ & $0$ & $(0)$\\ \hline
& \#2  & $A_2\oplus D_5 \subset E_8$ &   & $ E_8  E_8$ & $1$ & $(0)$\\ \hline
  $ D_{16}E_8$    &  $\mathbb Z /{2 \mathbb Z}$ & & & & &\\ \hline
&\#3   & $A_2\subset E_8$ & $D_5 \subset D_{16}$  & $E_6 D_{11}$ & $0$ & $(0)$\\ \hline
& \#4  & $A_2\oplus D_5 \subset E_8$ &   & $ D_{16}$ & $1$ & $\mathbb Z /{2 \mathbb Z}$\\ \hline
& \#5  & $D_5\subset E_8$ & $A_2 \subset D_{16}$  & $A_3  D_{13}$ & $1$ & $(0)$\\ \hline
& \#6  & $A_2\oplus D_5 \subset D_{16}$ &   & $E_8  D_{8}$ & $1$ & $(0)$\\ \hline
  $ D_{10}E_7^2$    &  $(\mathbb Z /{2 \mathbb Z})^2$ & & & & &\\ \hline
& \#7  & $A_2\subset E_7$ & $D_5 \subset D_{10}$  & $E_7 A_5 D_5$ & $0$ & $\mathbb Z /{2 \mathbb Z}$\\ \hline
&  \#8 & $A_2\subset E_7$ & $D_5 \subset E_7$  & $ A_5 A_1 D_{10}$ &$1$ & $\mathbb Z /{2 \mathbb Z}$\\ \hline
%&  & $A_2\oplus D_5 \subset E_7$ &   &  &  &\\ \hline
& \#9  & $A_2\oplus D_5 \subset D_{10}$ &   & $E_7 E_7 A_1 A_1$ & $1$ & $\mathbb Z /{2 \mathbb Z}$\\ \hline
& \#10  & $D_5\subset E_7$ & $A_2 \subset D_{10}$  & $A_1 D_7 E_7$ & $2$ & $(0)$ \\ \hline
 $ A_{17}E_7$    &  $\mathbb Z /{6 \mathbb Z}$ & & & & &\\ \hline
%&   & $A_2\oplus D_5 \subset E_7$ &   &  &  & \\ \hline
&\#11   & $D_5\subset E_7$ & $A_2 \subset A_{17}$  & $A_1 A_{14}$ & $2$ & $(0)$\\ \hline
 $D_{24}$    &  $\mathbb Z /{2 \mathbb Z}$ & & & & &\\ \hline
& \#12  & $A_2\oplus D_5 \subset D_{24}$ &   & $D_{16}$ & $1$ & $(0)$\\ \hline

 $D_{12}^2$    &  $(\mathbb Z /{2 \mathbb Z})^2$ & & & & &\\ \hline
&\#13   & $A_2 \subset D_{12}$ & $D_5 \subset D_{12}$  & $D_{9} D_{7}$ & $1$ & $(0)$ \\ \hline
&\#14   & $A_2\oplus D_5 \subset D_{12}$ &   & $D_4 D_{12}$ & $1$ & $\mathbb Z /{2 \mathbb Z}$\\ \hline
 $D_8^3$    &  $(\mathbb Z /{2 \mathbb Z})^3$ & & & & &\\ \hline
&\#15   & $A_2 \subset D_{8}$ & $D_5 \subset D_{8}$  & $ D_{5}A_3 D_{8}$ & $1$ &$\mathbb Z/{2\mathbb Z} $\\ \hline
& \#16  & $A_2\oplus D_5 \subset D_{8}$ &   & $ D_8 D_{8}$ & $1$ & $\mathbb Z /{2 \mathbb Z}$\\ \hline
 $ A_{15}D_9$    &  $\mathbb Z /{8 \mathbb Z}$ & & & & &\\ \hline
& \#17  & $A_2\oplus D_5 \subset D_{9}$ &   & $ A_{15}$ & $2$ &$\mathbb Z /{2 \mathbb Z}$ \\ \hline
& \#18  & $D_5 \subset D_{9}$ & $A_2 \subset A_{15}$  & $D_4 A_{12}$ & $1$ & $(0)$\\ \hline
 $E_6^4$    &  $(\mathbb Z /{3 \mathbb Z)^2}$ & & & & &\\ \hline
& \#19  & $A_2 \subset E_6$ & $D_5 \subset E_6$  & $A_2 A_2 E_6 E_6$ & $1$ & $\mathbb Z /{3 \mathbb Z}$\\ \hline
 $ A_{11} D_7E_6$    &  $\mathbb Z /{12 \mathbb Z}$ & & & & &\\ \hline
&\#20   & $A_2 \subset E_6$ & $D_5 \subset D_7$  & $A_2 A_2 A_1 A_1 A_{11}$ & $0$ & $\mathbb Z /{6 \mathbb Z}$\\ \hline
& \#21  & $A_2 \subset A_{11}$ & $D_5 \subset D_7$  & $A_8 A_1 A_1 E_6$ & $1$ & $(0)$\\ \hline
%&   & $A_2\oplus D_5 \subset D_{7}$ &   &  &  \\ \hline
& \#22  & $A_2 \subset A_{11}$ & $D_5 \subset E_6$  & $ A_8 D_7$ & $2$ & $(0)$\\ \hline
&\#23   & $D_5 \subset E_6$ & $A_2 \subset D_7$  & $ A_{11}  D_4$ & $2$ &$\mathbb Z /{2 \mathbb Z}$ \\ \hline
 $D_6^4$    &  $(\mathbb Z /{2 \mathbb Z})^4$ & & & & &\\ \hline
& \#24  & $A_2 \subset D_6$ & $D_5 \subset D_6$  & $ A_3 D_6 D_6$ & $2$ &$\mathbb Z /{2 \mathbb Z}$ \\ \hline
 $ A_9^2D_6$    &  $\mathbb Z /{2}\times \mathbb Z /{10} $ & & & & &\\ \hline
& \#25  & $D_5 \subset D_6$ & $A_2 \subset A_9$  & $ A_{6} A_9$ & $2$ & $(0)$\\ \hline

 $A_7^2D_5^2$    &  $\mathbb Z /{4}\times \mathbb Z /{8} $ & & & & &\\ \hline
&\#26   & $D_5 \subset D_5$ & $A_2 \subset D_5$  & $A_1  A_{1} A_7 A_7$ & $1$ &$\mathbb Z /{4 \mathbb Z}$ \\ \hline
&\#27   & $D_5 \subset D_5$ & $A_2 \subset A_7$  & $D_5 A_4 A_7 $ & $1$ & $(0)$\\ \hline

\end{tabular}
\end{center}
\caption{The elliptic fibrations of the Ap\'ery-Fermi family}\label{Ta:Fib2}
\end{table}
\begin{proof}

The torsion groups can be computed as explained in \cite{BL} or \cite{BGL}. Let us recall briefly the method.

Denote $\phi$ a primitive embedding of $M=A_2\oplus D_5$ into a Niemeier lattice $L$. Define $W=(\phi(M))_{L}^{\perp}$ and $N=(\phi(M))_{L_{\text{root}}}^{\perp}$. We observe that $W_{\text{root}}=N_{\text{root}}$. Thus computing $N$ then $N_{\text{root}}$ we know the type of singular fibers. Recall also that the torsion part of the Mordell-Weil group is 
\[ \overline{W_{\text{root}}}/W_{\text{root}} ( \subset W/N)\]
and can be computed in the following way \cite{BGL}: let $l+L_{\text{root}}$ be a non trivial element of $L/L_{\text{root}}$. If there exist $k\neq 0$ and $u\in L_{\text{root}}$ such that $k(l+u)\in N_{\text{root}}$, then $l+u \in W$ and the class of $l$ is a torsion element.

We use also several facts.
\begin{enumerate}
\item If the rank of the Mordell-Weil group is $0$, then the torsion group is equal to $W/N$.
Hence fibrations $\#1 ( A_3 E_6 E_8)$, $ \#3 (D_{11}E_6)$, $\#7 (D_5A_5E_7)$, $\#20 (A_{11}2A_1 2A_2)$ have respective torsion groups $(0)$, $(0)$, $\mathbb Z/2 \mathbb Z$, $\mathbb Z/ 6\mathbb Z$.

\item If there is a singular fiber of type $E_8$, then the torsion group is $(0)$. 
Hence the fibrations $\#1$, $\#2$ and $\#6$ have no torsion.

\item Using lemma \ref{lem:A} below and the shape of glue vectors we prove that fibrations $\#11$, $\#18$, $\#21$, $\#22$, $\#25$, $\#27$ have no torsion.

\begin{lemma} \label{lem:A}
Suppose $A_2$ primitively embedded in $A_n$, $A_2= \langle a_1,a_2 \rangle \hookrightarrow A_n$. Then for all $k\neq 0$, $k[1]_{A_n}\notin ((A_2)_{A_n}^{\perp})_{\text{root}}$.

\end{lemma}

\begin{proof}
It follows from the fact that $[1]_{A_n}$ is not orthogonal to $a_1$.
\end{proof}
  
\item Using lemma \ref{lem:B} below and the shape of glue vectors we can determine the torsion for elliptic fibrations $\#5$, $\#10$, $\#13$, $\#15$ $\#23$.

\begin{lemma}\label{lem:B}
Suppose $A_2$ primitively embedded in $D_l$, $A_2= \langle d_l,d_{l-2} \rangle \hookrightarrow D_l$. Then  $2.[2]_{D_l}\in ((A_2)_{D_l}^{\perp})_{\text{root}}$ but there is no $k$ satisfying  $k.[i]_{D_l}\in ((A_2)_{D_l}^{\perp})_{\text{root}}, i=1,3$.

\end{lemma}

\begin{proof}
It follows from Nishiyama \cite{Nis}:
\[(A_2)_{D_l}^{\perp}= \langle y,x_4,d_{l-4}, ..., d_1 \rangle \]
with $y=d_l+2d_{l-1}+2d_{l-2}+d_{l-3}$ and $x_4=d_l+d_{l-1}+2(d_{l-2}+d_{l-3}+...+d_2)+d_1$
and Gram matrix 
\[L_{l-3}^3=\begin{pmatrix}
                 - 4 & -1 &   0  &...  0\\
                       -1 &  \\
                      1 & D_{l-3} \\
                      . &  \\
                       0 &
                   \end{pmatrix}.
\]
Moreover $ ((A_2)_{D_l}^{\perp})_{\text{root}}= \langle x_4,d_{l-4},...,d_1 \rangle $.
From there we compute easily the relation $2.[2]_{D_l}=x_4+d_{l-4}+2(d_{l-5}+...+d_1)$.
The last assertion follows from the fact that $[i]_{D_l}$ is not orthogonal to 
$A_2$.
\end{proof}
\end{enumerate}

We now give some examples showing the method in detail.

\subsubsection{Fibration $\#17$}
It comes from a primitive embedding of $A_2\oplus D_5$ into $D_9$ giving a primitive embedding of $A_2\oplus D_5$ into $Ni(A_{15}D_9)$ with glue code $ \langle [2,1] \rangle $. Since by lemma \ref{lem:2.1}(2) $N_{\text{root}}=A_{15}$, among the elements $k.[2,1]$, only $4.[2,1]=[8,4.1\in D_9]$ satisfies $2.[8,0+u]\in N_{\text{root}}=A_{15}$ with $u=4.1$. Hence the torsion group is $\mathbb Z/2 \mathbb Z$.

\subsubsection{Fibration $\#19$}
It comes from a primitive embedding of $A_2= \langle e_1,e_3 \rangle $ into $E_6^{(1)}$ and $D_5= \langle e_2,e_3,e_4,e_5,e_6 \rangle $ into $E_6^{(2)}$ giving a primitive embedding of $A_2 \oplus D_5$ into $Ni(E_6^4)$. In that case $Ni(E_6^4)/E_6^4\simeq (\mathbb Z/3 \mathbb Z)^2$ and the glue code is $ \langle [1,(0,1,2)] \rangle $. Moreover $(D_5)_{E_6}^{\perp}=3e_2+4e_1+5e_3+6e_4+4e_5+2e_6=a$, $(A_2)_{E_6}^{\perp}= \langle e_2,y \rangle \oplus \langle e_5,e_6 \rangle $ with $y=2e_2+e_1+2e_3+3e_4+2e_5+e_6$. From the relation
\[[1]_{E_6}=-\frac{1}{3}(2e_1+3e_2+4e_3+6e_4+5e_5+4e_6)\] 
we get
\[-3.[1]_{E_6}=a-2e_1-e_3+e_5+2e_6\in E_6\]
\[-3.[1]_{E_6}=2y-e_2+e_5+2e_6\in (A_2)_{E_6}^{\perp}\]
we deduce that only $[1,0,1,2]$, $[2,0,2,1]$, $[0,0,0,0]$ contribute to the torsion thus the torsion group is $\mathbb Z/3 \mathbb Z$.

\subsubsection{Fibration $\#10$}
The  embeddings of $A_2= \langle d_{10},d_8 \rangle $ into $D_{10}$ and $D_5= \langle e_2,e_3,e_4,e_5,e_6 \rangle $ into $E_7^{(1)}$ lead to a primitive embedding of $A_2\oplus D_5$ into $Ni(D_{10}E_7^2)$ satisfying $Ni(D_{10}E_7^2)/(D_{10}E_7^2)\simeq (\mathbb Z/2 \mathbb Z)^2$ with glue code 
$ \langle [1,1,0], [3,0,1] \rangle $. We deduce from lemma \ref{lem:B} that no glue vector can contribute to the torsion which is therefore $(0)$.

\subsubsection{Fibration $\#18$}
 
The embeddings of $A_2= \langle a_1,a_2 \rangle $ into $A_{15}$ and $D_5= \langle d_9,d_7,d_8,d_6,d_5 \rangle $ into $D_9$ lead to a primitive embedding of $A_2\oplus D_5$ into $Ni(A_{15}D_9)$ satisfying $Ni(A_{15}D_9)/(A_{15}D_9)\simeq (\mathbb Z/8 \mathbb Z)$ with glue code $ \langle [2,1] \rangle 
$. We deduce from lemma \ref{lem:A} that no glue vector can contribute to the torsion which is therefore $(0)$.

\subsubsection{Fibration $\#8$}

The primitive embeddings of $A_2= \langle e_1,e_3 \rangle $ into $E_7^{(1)}$ and $D_5= \langle e_2,e_3,e_4,e_5,e_6 \rangle $ into $E_7^{(2)}$ lead to a primitive embedding of $A_2\oplus D_5$ into $Ni(D_{10}E_7^2)$ satisfying $Ni(D_{10}E_7^2)/(D_{10}E_7^2) \simeq (\mathbb Z/2 \mathbb Z)^2$ with glue code $ \langle [1,1,0], [3,0,1] \rangle $. From Nishiyama \cite{Nis} we get $(A_2)_{E_7^{(1)}}^{\perp}= \langle e_2,y,e_7,e_6,e_5 \rangle \simeq A_5$ with $y=2e_2+e_1+2e_3+3e_4+2e_5+e_6$ and $(D_5)_{E_7}^{\perp}= \langle (-4),e_2+e_3+2(e_4+e_5+e_6+e_7)=(-2) \rangle $. Hence $N=D_{10}\oplus A_5 \oplus (-4)\oplus A_1$ and $W_{\text{root}}=N_{\text{root}}=D_{10}\oplus A_5 \oplus A_1$. Now
\[-2\eta_7=-2.[1]_{E_7}=2y-e_2+e_5+2e_6+3e_7 \in ((A_2)_{E_7}^{\perp})_{\text{root}}\]
and for all $k\neq 0$, $k.[1]_{E_7}\notin (D_5)_{E_7}^{\perp}$. Hence only the generator $[1,1,0]$ can contribute to the torsion group which is therefore $\mathbb Z/2 \mathbb Z$.

\subsubsection{Fibration $\#24$}
 The primitive embeddings $A_2= \langle d_6,d_4 \rangle $ into $D_6^{(1)}$ and $D_5= \langle d_6,d_5, d_4, d_3 ,d_2 \rangle $ into $D_6^{(2)}$ give a primitive embedding of $A_2 \oplus D_5$ into $L=Ni(D_6^4)$ with $L/L_{\text{root}}\simeq (\mathbb Z/2 \mathbb Z)^4$ and glue code $ \langle \text{ even permutations of } [0,1,2,3] \rangle $. From Nishiyama \cite{Nis} we get $(A_2)_{D_6}^{\perp}= \langle y=2d_5+d_6+2d_4+d_3,x_4=d_5+d_6+2(d_4+d_3)+d_2,d_2,d_1 \rangle $,  $((A_2)_{D_6}^{\perp})_{\text{root}}= \langle x_4,d_2,d_1 \rangle \simeq A_3$  and $(D_5)_{D_6}^{\perp}= \langle x'_6 \rangle = \langle d_5+d_6+2(d_4+d_3+d_2+d_1)=(-4) \rangle $. We deduce $N_{\text{root}}=A_3\oplus D_6 \oplus D_6$. From the relations $2.[2]_{D_6}=x_4+d_2+2d_1$ and $2.[3]_{D_6}=y+x_4+d_2+d_1$  we deduce that the glue vectors having $1$, $2$, $3$ or $0$ in the first position may belong to $W$. From the relation $2.[2]_{D_6}=x'_6$ we deduce that only glue vectors with $2$ or $0$ in the second position may belong to $W$. Finally only the glue vectors $[0,2,3,1], [1,0,3,2], [1,2,0,3], [2,0,1,3], [2,2,2,2], [3,0,2,1], [3,2,1,0], [0,0,0,0] $ belong to $W$. Since $y$ and $x'_6$ are not roots, only glue vectors with $0$ or $2$ in the first position and $0$ in the second position may contribute to torsion that is  
 $ [2,0,1,3], [0,0,0,0] $. Hence the torsion group is $\mathbb Z/2 \mathbb Z$.

\subsubsection{Fibration $\#26$}
The primitive embeddings of $A_2= \langle d_5,d_3 \rangle $ into $D_5^{(1)}$ and $D_5$ into $D_5^{(2)}$ give a primitive embedding into $L=Ni(A_7^2D_5^2)$ with $L/L_{\text{root}}\simeq \mathbb Z/8 \mathbb Z \times \mathbb Z/4 \mathbb Z$ and glue code $ \langle [1,1,1,2], [1,7,2,1] \rangle $. From Nishiyama we get $(A_2)_{D_5}^{\perp}= \langle y,x_4,d_1 \rangle $ with $y=2d_4+d_5+2d_3+d_2$, $x_4=d_5+d_4+2d_3+2d_2+d_1$ and Gram matrix $M_2^4=\begin{pmatrix}
                       -4 & -1 & 1 \\
                       -1 & -2 & 0 \\
                        1 & 0 & -2
                       \end{pmatrix}$
of determinant $12$. We also deduce $N_{\text{root}}=E_7^2 A_1^2$, $2.[2]_{D_5}=x_4+d_1\in ((A_2)_{D_5}^{\perp})_{\text{root}}$. Moreover neither $k.[1]_{D_5}$ nor  $k.[3]_{D_5}$ belongs to $(A_2)_{D_5}^{\perp}$. Thus only glue vectors with $2$ or $0$ in the third position can belong to $W$ and eventually contribute to torsion,
that is $[2,2,2,0]$, $[4,4,0,0]$, $[6,6,2,0]$, $[2,6,0,2]$, $[6,2,0,2]$, $[4,0,2,2]$, $[0,4,2,2]$, $[0,0,0,0]$. Since there is no $u_4\in D_5$ satisfying $2.(2+u_4)=0$ or $4.(2+u_4)=0$, glue vectors with the last component equal to $2$ cannot satisfy $k(l+u)\in N_{\text{root}}$ with $l\in L$ and $u\in L_{\text{root}}=A_7^2D_5^2$. Hence only the glue vectors generated by $ \langle [2,2,2,0] \rangle $ contribute to torsion and the torsion group is therefore $\mathbb Z/4 \mathbb Z$.

\end{proof}

%We consider the pencil of $K3$ surfaces
%\[
%Y_{k}:X+\frac{1}{X}+Y+\frac{1}{Y}+Z+\frac{1}{Z}=k
%\]
%where the base of the pencil can be identified with the modular curve
%$X_{0}\left(  6\right)  /<w_{2},w_{3}>.$

%Using the elliptic fibration
%\begin{align*}
%Y_{k}  &  \rightarrow\mathbb{P}^{1}\\
%\left(  X,Y,Z\right)   &  \mapsto w=\left(  X+Y+Z\right)
%\end{align*}
%we can see that these surfaces have generic Picard rank $19$ and determinant
%of transcendant lattice equal to $12.$

%If $k=s+\frac{1}{s}$ we can construct an elliptic fibration with $2E_{8}$ and
%a section of height $12$

\section{Weierstrass Equations for all the elliptic fibrations of $Y_k$}

The method can be found in \cite{BL}, \cite{El}. We follow also the same kind of computations
used for $Y_{2}$ given in \cite{BL}.  We give only explicit computations for $4$ examples, \#19, \#2, \#9, and \#16. For \#2 and \#9 it was not obvious
to find a rational point on the quartic curve. All the results are given in Table \ref{Ta:W-Eq}.
For the $2$ or $3$-neighbor method \cite{El} we give in the third column  the
starting fibration and in the forth the elliptic parameter. The terms in the elliptic parameter refer to the starting fibration.

\subsection{Fibration \#19} %: $2IV^{\ast}$}

We take $u=\frac{XY}{Z}$ as a parameter of an elliptic fibration and 
with the birational transformation
\[
x=-u(1+uZ)(u+Y),\quad y=u^{2}\left(  (u+Y)(uY-1)Z+Y\left(  Y+2u+k\right)
-1\right)
\]

we obtain a
Weierstrass equation
\[
y^{2}+ukyx+u^{2}\left(  u^{2}+uk+1\right)  y=x^{3},
\]
where the point $\left(  x=0,y=0\right)  $ is a $3$-torsion point and the
point  $\left(  -u^{2},-u^{2}\right)$ is of  infinite order.

The singular fibers are of type $IV^{\ast}\left(  u=0,\infty\right)$,   
$I_{3}\left(  u^{2}+uk+1=0\right)$ and  

 $I_{1}\left(  27u^{2}-k(k^{2}
-27)u+27=0\right).$ 
Moreover if $k=s+\frac{1}{s\text{ }}$ the two singular
fibers of type $I_{3}$ are above $u=-s$ and $\frac{-1}{s}$.

\subsection{Fibration \#2} %: $2II^{\ast}$}

Using the $3$-neighbor method from fibration \#19 we construct a new fibration
with a fiber of type $II^{\ast}$ and the parameter $m=\frac{ys}{\left(
u+s\right)  ^{2}}.$ Then we obtain a cubic $C_{m}$ in $w,u,$ with 
$x=w\left(  u+s\right)  $%
\[
C_{m}:\left(  s+u\right)  m^{2}+u\left(  s^{2}w+u^{2}s+w+u\right)
m-w^{3}s^{2}=0.
\]
From some component of the fiber of type $I_{3}$ at $u=-s$ we obtain the rational
point on $C_{m}:\omega_{m}=\left(  u_{1}=\frac{ms-1}{s-m},w_{1}=\frac{m\left(
s^{2}-1\right)  }{s\left(  s-m\right)  }\right)  $ which is not a flex point.
The first stage is to obtain a quartic equation $Qua:y^{2}=ax^{4}%
+bx^{3}+cx^{2}+dx+e^{2}.$ First we observe  that $\omega_{m}$ is on the line
$w=u+\frac{1}{s},$ so we replace $w$ by $K$ with $w=u+\frac{1}{s}+K$ and
$u=u_{1}+T$.  The transformation $K=WT$ gives an equation of degree two in
$T$,  with constant term $fW+g$ where $f$ and $g$ belong to $\mathbb{Q}(s,m)$. With the change variable  \ $Wf+g=x$ \ we have an
equation $M(x)T^{2}+N(x)T+x=0.$ The discriminant of the quadratic equation in $T$ is $N(x)^2-4xM(x),$ a polynomial of degree $4$ in $x$ and constant term a square.
 Easily we obtain the form $Qua.$

From the quartic form, setting  $y=e+\frac{dx}%
{2e}+x^{2}X^{\prime}, \,x=\frac{8c^{3}X^{\prime}-4ce^{2}+d^{2}}{Y^{\prime}}$
$\ $\ we get
\[
Y^{\prime2}+4e\left(  dX^{\prime}-be\right)  Y+4e^{2}\left(  8e^{3}X^{\prime
}-4ce^{2}+d^{2}\right)  \left(  X^{\prime2}-a\right)  =0.
\]
Finally the following Weierstrass equation follows from standard transformation 
 where we
replace $m$ by $t$%

\begin{align*}
&  Y^{2}-X^{3}+\frac{1}{3}t^{4}(s^{2}+1)(s^{6}+219s^{4}-21s^{2}+1)X\\
&  -\frac{2t^{5}}{27}\left(  -864s^{5}t^{2}+(s^{4}+14s^{2}+1)(s^{8}%
-548s^{6}+198s^{4}-44s^{3}+1)t-864s^{5}\right)  =0,
\end{align*}
with a section $\Phi$ of height $12$ corresponding to $\left(  8e^{3}X^{\prime
}-4ce^{2}+d^{2}\right)  =0$ and $Y^{\prime}=0.$ The coordinates of $\Phi,$ too
long, are omitted but we can follow the previous computation to obtain it.  

Writing the above form as %
\[
y^{2}=x^{3}-3\alpha x+\left(  t+\frac{1}{t}\right)  -2\beta
\]
we recover the values of the $j$ invariants of the two elliptic curves for the Shioda-Inose structure
(see paragraph \ref{2} and \ref{3} below).

\subsection{Fibration \#9}%: $2III^{\ast}$}

Let $g=\frac{XY}{Z^{2}}.$  Eliminating $X$ and writing $Y=ZU$ we obtain an equation
of bidegree $2$ in $U$ and $Z.$ If $k=s+\frac{1}{s}$ there is a rational
point $U=-1,$ $Z=-\frac{s}{g}$ on the previous curve. By standard transformations we get a Weierstrass equation
\[
y^{2}=x^{3}+\frac{1}{4}g^{2}\left(  s^{4}+14s^{2}+1\right)  x^{2}+s^{2}%
g^{3}\left(  g+s^{2}\right)  \left(  gs^{2}+1\right)  x
\]
and a rational point
\[
x=\frac{s^{2}\left(  g-1\right)  ^{2}\left(  g+s^{2}\right)  \left(
s^{2}g+1\right)  }{\left(  s^{2}-1\right)  ^{2}},
\]
\[
y=\frac{1}{2}\frac
{s^{2}\left(  g^{2}-1\right)  \left(  g+s^{2}\right)  \left(  s^{2}g+1\right)
\left(  2g^{2}s^{2}+g\left(  s^{4}-6s^{2}+1\right)  +2s^{2}\right)  }{\left(
s^{2}-1\right)  ^{3}}.%
\]
The singular fibers are of type $2III^{\ast}\left(  \infty,0\right)
, \;2I_{2}\left(  -s^{2},-\frac{1}{s^{2}}\right)  , \;4I_{1}.$

\subsection{Fibration \#16}%:$2I_{4}^{\ast}$.}

Using the  fibration \#9 we consider the parameter $t=\frac{x}{g\left(
g+s^{2}\right)  }$ and obtain a Weierstrass equation
\[
Y^{2}=X^{3}+\left(  4t\left(  t^{2}+s^{2}\right)  +t^{2}\left(  s^{4}%
+14s^{2}+1\right)  \right)  X^{2}+16s^{6}t^{4}X.
\]
The singular fibers are of type $I_{4}^{\ast}\left(  \infty,0\right)  ,\;4I_{1}.$

\begin{table}[tp]\footnotesize
\[%
\begin{array}
[c]{|c|c|c|c|}%
\hline
& \text{Weierstrass Equation} &  \text{From} & \text{Param.}\\
\hline
\#1 &
\begin{array}
[c]{c}%
y^{2}+tkyx+t^{2}k\left(  t+1\right)  y=x^{3}-t^{4}\left(  t+1\right)  ^{3}\\
\hline
II^{\ast}\left(  \infty\right)  ,IV^{\ast}\left(  0\right)  ,I_{4}\left(
-1\right)  ,2I_{1}\\
\hline
r=0
\end{array}
&    &\scriptstyle{ \frac{Y \left(  X+Z\right)  ^{2}\left(  Z+Y\right)  }{XZ^{3}}}\\
\hline \hline
\#2 &
\begin{array}
[c]{c}%
y^{2}=x^{3}-\frac{1}{3}t^{4}\left(  s^{2}+1\right)  \left(  s^{6}+219\ast
s^{4}-21s^{2}+1\right)  x\\
\scriptstyle{+\frac{2}{27}t^{5}(-864s^{5}t^{2}+(s^{4}+14s^{2}+1)(s^{8}-548s^{6}%
+198s^{4}-44s^{2}+1)t-864s^{5})}\\
\hline
2II^{\ast}\left(  \infty,0\right)  ,4I_{1}\\
\hline
x_{P}=\Phi
\end{array}
& \#19 & \frac{ys}{\left(  s+t\right)  ^{2}}\\
\hline \hline
\#3 &
\begin{array}
[c]{c}%
y^{2}=x^{3}+\frac{1}{4}t\left(  4t^{2}s^{4}+\left(  s^{4}-10s^{2}+1\right)
t+12\right)  x^{2}\\
-t^{2}\left(  2ts^{2}-3\right)  x+t^{3}\\
\hline
I_{7}^{\ast}\left(  \infty\right)  ,IV^{\ast}\left(  0\right)  ,3I_{1}\\
\hline
r=0\\%
\end{array}
&   \#7 & \frac{x}{s^{4}t^{2}}\\
\hline \hline
\#4 &
\begin{array}
[c]{c}%
y^{2}=x^{3}+\\
\scriptstyle{\left(  \frac{1}{2}t^{3}-\frac{1}{24}\left(  s^{2}+1\right)  \left(
s^{6}+219s^{4}-21s^{2}+1\right)  t 
+\frac{1}{216}\left(  s^{8}-548s^{6}%
+198s^{4}-44s^{2}+1\right)  \left(  s^{4}+14s^{2}+1\right)  \right)  x^{2}}\\
+16s^{10}x\\
\hline
I_{12}^{\ast}\left(  \infty\right)  ,6I_{1}\\
\hline

x_P
\end{array}
&   \#2 & \frac{x}{2t^{2}}\\
\hline \hline
\#5 &
\begin{array}
[c]{c}%
y^{2}-k\left(  t+1\right)  yx+ky=x^{3}+\left(  t^{3}-3\right)  x^{2}+3x-1\\
\hline
I_{9}^{\ast}\left(  \infty\right)  ,I_{4}\left(  0\right)  ,5I_{1}\\
\hline
x_{P}=0
\end{array}
&   \#1 & \frac{x}{t^{2}}\\
\hline  \hline  
\#6 &
\begin{array}
[c]{c}%
y^{2}=x^{3}+\left(  \frac{1}{4}t^{2}\left(  s^{4}+14s^{2}+1\right)
+t^{3}s^{2}\right)  x^{2}\\
+t^{4}s^{2}\left(  s^{4}+1\right)  x+t^{5}s^{6}\\
\hline
I_{4}^{\ast}\left(  \infty\right)  ,II^{\ast}\left(  0\right)  ,4I_{1}\\
\hline
x_{P}
\end{array}
&   \#9 & \frac{x}{\left(  t+s^{2}\right)  \left(  ts^{2}+1\right)  }\\
\hline \hline
\#7 &
\begin{array}
[c]{c}%
y^{2}=x^{3}+\frac{1}{4}t\left(  t\left(  s^{4}-10s^{2}+1\right)
+8s^{4}\right)  x^{2}-t^{2}s^{2}\left(  t-s^{2}\right)  ^{3}x\\
\hline
III^{\ast}\left(  \infty\right)  ,I_{1}^{\ast}\left(  0\right)  ,I_{6}\left(
s^{2}\right)  ,2I_{1}\\
\hline%
r=0
\end{array}
&   \#15 & \frac{x}{t}\\
\hline  \hline
\#8 &
\begin{array}
[c]{c}%
y^{2}-k\left(  t-1\right)  yx=x\left(  x-1\right)  \left(  x-t^{3}\right)  \\
\hline
I_{6}^{\ast}\left(  \infty\right)  ,I_{6}\left(  0\right)  ,I_{2}\left(
1\right)  ,4I_{1}\\
\hline
x_{P}=1
\end{array}
&    & \frac{\left(  X+Z\right)  \left(  Y+Z\right)  }{XZ}\\
\hline  \hline
\#9 &
\begin{array}
[c]{c}%
y^{2}=x^{3}+\frac{1}{4}t^{2}\left(  s^{4}+14s^{2}+1\right)  x^{2}+t^{3}%
s^{2}\left(  t+s^{2}\right)  \left(  ts^{2}+1\right)  x\\
\hline
2III^{\ast}\left(  \infty,0\right)  ,2I_{2}\left(  -s^{2},-\frac{1}{s^{2}%
}\right)  ,2I_{1}\\
\hline
x_{P}=\frac{s^{2}\left(  t-1\right)  ^{2}\left(  t+s^{2}\right)  \left(
ts^{2}+1\right)  }{\left(  s^{2}-1\right)  ^{2}}%
\end{array}
&    & \frac{XY}{Z^{2}}\\
\hline \hline
\#10 &
\begin{array}
[c]{c}%
y^{2}+t\left(  s^{2}+1\right)  \left(  x+t^{2}s^{2}\right)  y=\left(
x-t^{3}s^{2}\right)  \left(  x^{2}+t^{3}s^{4}\right)  \\
\hline
I_{3}^{\ast}\left(  \infty\right)  ,III^{\ast}\left(  0\right)  ,I_{2}\left(
-1\right)  ,4I_{1}\\
\hline
x_{P_{1}}=t^{3}s^{2}, \quad x_{P_{2}}=0
\end{array}
&    & \frac{XY}{\left(  Y+Z\right)  Z}\\
\hline  \hline
\#11 &
\begin{array}
[c]{c}%
y^{2}+t\left(  st-1-s^{2}\right)  yx-s^{3}y=x^{2}\left(  x+s\left(  t\left(
s^{2}-1\right)  -s\left(  s^{2}+1\right)  \right)  \right)  \\
\hline
I_{15}\left(  \infty\right)  ,I_{2}\left(  s\right)  ,7I_{1}\\
\hline
x_{P_{1}}=st, \quad x_{P_{2}}=-s^{3}t+s^{2}\left(  s^{2}+1\right)
\end{array}
&   \#8 & \frac{y+sx}{xt}\\
\hline \hline
\#12 &
\begin{array}
[c]{c}%
y^{2}=x^{3}+t\left(  t^{2}s^{2}+\frac{1}{4}\left(  s^{4}+14s^{2}+1\right)
+\left(  s^{4}+1\right)  \right)  x^{2}\\
\scriptstyle{-\left(  2t^{2}s^{4}+\frac{1}{2}s^{2}\left(  s^{4}+14s^{2}+1\right)
t+s^{2}\left(  s^{4}+1\right)  \right)  x+ts^{6}+\frac{1}{4}s^{4}\left(
s^{4}+14s^{2}+1\right)  }\\
\hline
I_{12}^{\ast}\left(  \infty\right)  ,6I_{1}\\
\hline
x_{P}
\end{array}
&   \#14 & \frac{x}{t^{2}}+\frac{s^{2}}{t}\\
\hline  \hline
\#13 &
\begin{array}
[c]{c}%
y^{2}=x^{3}+\frac{1}{4}t\left(  4t^{2}+\left(  s^{4}-10s^{2}+1\right)
t+4s^{4}\right)  x^{2}\\
+2t^{3}s^{4}\left(  t-s^{4}\right)  x+t^{5}s^{8}\\
\hline
I_{5}^{\ast}\left(  \infty\right)  ,I_{3}^{\ast}\left(  0\right)  ,4I_{1}\\
\hline
x_{P}=-ts^{4}%
\end{array}
&   \#15 & \frac{x}{t^{2}}\\
\hline 
\end{array}
\]
\caption{Weierstrass equations of the elliptic fibrations of $Y_k$}\label{Ta:W-Eq}
\end{table}

\begin{table}[tp]\footnotesize
\[%
\begin{array}
[c]{|c|c|c|c|}%
\hline
No & \text{Weierstrass Equation } & \text{From} & \text{Param.}\\
\hline
\#14 &
\begin{array}
[c]{c}%
y^{2}=x^{3}+\left(  t^{3}\left(  s^{4}+1\right)  +\frac{1}{4}t^{2}\left(
s^{4}+14s^{2}+1\right)  +ts^{2}\right)  x^{2}+s^{4}t^{6}x\\
\hline
I_{0}^{\ast}\left(  \infty\right)  ,I_{8}^{\ast}\left(  0\right)
,4I_{1}\left(  -\frac{1}{4},-4\frac{s^{2}}{\left(  s^{2}-1\right)  ^{2}%
},..\right)  \\
\hline
x_{P}=\frac{s^{4}\left(  2t+1\right)  ^{2}}{\left(  s^{2}-1\right)  ^{2}}%
\end{array}
&   \#9 & \frac{x}{t\left(  t+s^{2}\right)  \left(  ts^{2}+1\right)  }\\
\hline  \hline
\#15 &
\begin{array}
[c]{c}%
\left(  y-tx\right)  \left(  y-s^{2}tx\right)  =x\left(  x-ts^{2}\right)
\left(  x-ts^{2}\left(  t+1\right)  ^{2}\right)  \\
\hline
I_{4}^{\ast}\left(  \infty\right)  ,I_{1}^{\ast}\left(  0\right)
,I_{4}\left(  -1\right)  ,3I_{1}\left(  \frac{1}{4}\left(  \frac{s^{2}-1}%
{s}\right)  ^{2},..\right)  \\
\hline
x_{P}=s^{2}t
\end{array}
&    & \frac{\left(  XY+1\right)  Z}{X}\\
\hline  \hline
\#16 &
\begin{array}
[c]{c}%
y^{2}=x^{3}+t\left(  4\left(  t^{2}+s^{2}\right)  +t\left(  s^{4}%
+14s^{2}+1\right)  \right)  x^{2}+16s^{6}t^{4}x\\
\hline
I_{4}^{\ast}\left(  \infty,0\right)  ,4I_{1}\\
\hline
x_{P}=\frac{-4ts^{6}\left(  t+1\right)  ^{2}}{\left(  t+s^{2}\right)  ^{2}}%
\end{array}
&   \#9 & \frac{x}{t\left(  t+s^{2}\right)  }\\%
\hline\hline
\#17 &
\begin{array}
[c]{c}%
y^{2}-\frac{1}{2}\left(  s^{4}+14s^{2}+1-s^{2}t^{2}\right)  yx=x\left(
x-4s^{2}\right)  \left(  x-4s^{6}\right)  \\
\hline
I_{16}\left(  \infty\right)  ,8I_{1}\left(  \pm\frac{s^{2}\pm4s-1}%
{s},..\right)  \\
\hline
x_{P_{1}}=4s^{2};\quad x_{P_{2}}=\frac{4s^{4}\left(  ts+s^{2}-1\right)  ^{2}%
}{\left(  ts+1-s^{2}\right)  ^{2}}%
\end{array}
& \#16 & \frac{y}{ts^{2}x}\\
\hline \hline
\#18 &
\begin{array}
[c]{c}%
y^{2}+\left(  -t^{2}+\left(  s^{2}-1\right)  t-2s^{2}\right)  yx+s^{4}%
t^{2}y=x^{2}\left(  x-s^{4}\right)  \\
\hline
I_{13}\left(  \infty\right)  ,I_{0}^{\ast}\left(  0\right)  ,5I_{1}\\
\hline
x_{P}=0
\end{array}
& \#15 & \frac{y-tx}{t\left(  x-ts^{2}\right)  }\\
\hline  \hline
\#19 &
\begin{array}
[c]{c}%
y^{2}+ktyx+t^{2}\left(  t^{2}+tk+1\right)  y=x^{3}\\
\hline
2IV^{\ast}\left(  \infty,0\right)  ,2I_{3}\left(  -s,-\frac{1}{s}\right)
,2I_{1}\\
\hline
x_{P}=-t^{2}%
\end{array}
&  & \frac{XY}{Z}\\
\hline \hline
\#20 &
\begin{array}
[c]{c}%
y^{2}-yx\left(  t^{2}-kt+1\right)  =x\left(  x-1\right)  \left(
x+t^{2}-tk\right)  \\
\hline
I_{12}\left(  \infty\right)  ,2I_{2}\left(  0,k\right)  ,4I_{1}\left(
s,\frac{1}{s},..\right)  \\
\hline
r=0
\end{array}
&  & X+Y+Z\\
\hline  \hline
\#21 &
\begin{array}
[c]{c}%
y^{2}=x^{3}+\frac{1}{4}t^{2}\left(  t^{2}+2\left(  s^{2}-1\right)  t+\left(
s^{4}-10s^{2}+1\right)  \right)  x^{2}\\
\hline
+\frac{1}{2}t^{3}s^{4}\left(  t-\left(  s^{2}-1\right)  \right)  x+\frac{1}%
{4}s^{8}t^{4}\\
\hline
I_{9}\left(  \infty\right)  ,IV^{\ast}\left(  0\right)  ,2I_{2}\left(
1,-s^{2}\right)  ,3I_{1}\\
\hline
x_{P}=s^{2}t^{2}%
\end{array}
& \#15 & \frac{y-s^{2}x-ts^{4}\left(  t^{2}-1\right)  }{x-ts^{2}(t+1)^{2}}\\
\hline  \hline 
\#22 &
\begin{array}
[c]{c}%
y^{2}+\left(  t\left(  1-s^{2}\right)  +s^{2}\right)  yx+t^{3}s^{2}y=x\left(
x-s^{2}t\right)  \left(  x+t^{2}s^{2}\left(  1-t\right)  \right)  \\
\hline
I_{3}^{\ast}\left(  \infty\right)  ,I_{9}\left(  0\right)  ,6I_{1}\\
\hline
x_{P_{1}}=1,\quad x_{P_{2}}=s^{2}t
\end{array}
&  & \frac{Z\left(  XYZ+s\right)  }{1+YZ}\\
\hline \hline
\#23 &
\begin{array}
[c]{c}%
y^{2}+\left(  2t^{2}-tk+1\right)  yx=x\left(  x-t^{2}\right)  \left(
x-t^{4}\right)  \\
\hline
I_{0}^{\ast}\left(  \infty\right)  ,I_{12}\left(  0\right)  ,6I_{1}\left(
\frac{1}{k\pm2},..\right)  \\
\hline
x_{P_{1}}=t^{2},\quad x_{p_{2}}=\frac{\left(  tk-1\right)  ^{2}}{k^{2}-4}%
\end{array}
&  & \frac{1}{X+Y}\\
\hline  \hline
\#24 &
\begin{array}
[c]{c}%
y^{2}+\left(  s^{2}+1\right)  tyx=x\left(  x-t^{2}s^{2}\right)  \left(
x-s^{2}t\left(  t+1\right)  ^{2}\right)  \\
\hline
2I_{2}^{\ast}\left(  \infty,0\right)  ,I_{4}\left(  -1\right)  ,4I_{1}\\
\hline
x_{P_{1}}=t+1;\quad x_{P_{2}}=t^{2}s^{2}%
\end{array}
&  & \frac{Z}{Y}\\
\hline \hline
\#25 &
\begin{array}
[c]{c}%
y^{2}+\left(  s+t\right)  \left(  ts+1\right)  yx+t^{2}s^{2}\left(  t\left(
s^{2}-1\right)  +s\right)  y=\\
x\left(  x-st\right)  \left(  x-t^{2}s\left(  t-s\right)  \right)  \\
\hline
I_{7}(\infty),I_{10}(0),7I_1\\
\hline
x_{P_{1}}=ts;\quad x_{P_{2}}=-t^{2}s^{2}%
\end{array}
&  & \frac{Y-s}{XY+sZ}\\
\hline \hline
\#26 &
\begin{array}
[c]{c}%
y^{2}+(ts-1)(t-s)xy=x\left(  x-t^{2}s^{2}\right)  ^{2}\\
\hline
I_{8}\left(  \infty,0\right)  ,I_{2}\left(  s,\frac{1}{s}\right)  ,4I_{1}\\
\hline
x_{P}=ts
\end{array}
&  & Z\\
\hline \hline
\#27 &
\begin{array}
[c]{c}%
y^{2}-\left(  t\left(  s^{2}-1\right)  +s^{2}\right)  yx+t^{3}s^{2}\left(
t+1\right)  y=\\
x^{2}\left(  x+t^{2}s^{2}\left(  t+1\right)  \right)  \\
\hline
I_{1}^{\ast}\left(  \infty\right)  ,I_{8}\left(  0\right)  ,I_{5}\left(
-1\right)  4I_{1}\\
\hline
x_{P}=0
\end{array}
&  & \frac{Z-s}{X+Y}\\
\hline
\end{array}
\]
\caption{Weierstrass equations of the elliptic fibrations of $Y_k $}\label{Ta:W-Eq1}
\end{table}

\section{Nikulin involutions and Shioda-Inose structure}

\subsection{Background}
Let $X$ be a $K3$ surface. 

The second cohomology group, $H^{2}\left(X,\mathbb{Z}\right)  $ equipped with the cup product is an even unimodular
lattice of signature $\left(  3,19\right) $. The period lattice of a surface
denoted $T_{X}$ is defined by
\[
T_{X}=S_{X}^{\bot}\subset H^{2}\left(  X,\mathbb{Z}\right)
\]
where $S_{X}$ is the N\'{e}ron-Severi group of $X.$ The lattice $H^{2}\left(
X,\mathbb{Z}\right)  $ admits a Hodge decomposition of weight two%
\[
H^{2}\left(  X,\mathbb{C}\right)  \simeq H^{2,0}\mathbb{\oplus}H^{1,1}%
\mathbb{\oplus}H^{0,2}.%
\]
Similarly, the period lattice $T_{X}$ has a Hodge decomposition of weight two
\[
T_{X}\otimes\mathbb{C}\simeq T^{2,0}\mathbb{\oplus}T^{1,1}\mathbb{\oplus
}T^{0,2}.
\]
An isomorphism between two lattices that preserves their bilinear forms and
their Hodge decomposition is called a Hodge isometry.

An automorphism of a $K3$ surface $X$ is called \textit{symplectic} if it acts
on $H^{2,0}(X)$ trivially. Such automorphisms were studied by Nikulin in \cite{N1}
who proved that a symplectic involution $i$ (\textit{Nikulin involution}) has eight fixed points and that the
minimal resolution $Y\rightarrow X/\langle i \rangle$ of the eight nodes is again a $K3$ surface.

We have then the rational quotient map $p:X\rightarrow Y$ of degree 2. The
transcendental lattices $T_{X}$ and $T_{Y}$ are related by the chain of
inclusions
\[
2T_{Y}\subseteq p^{\ast}T_{X}=T_{X}(2)\subseteq T_{Y},
\]
which preserves the quadratic forms and the Hodge structures.

In this paper,  $K3$ surfaces  are given as elliptic surfaces. If we have a
2-torsion section $\tau$, we consider the symplectic involution $i$ (\textit{Van Geemen-Sarti involution}) given by
the fiberwise translation by $\tau$. In this situation, the rational quotient
map $X\rightarrow Y$ is just an isogeny of degree 2 between elliptic curves
over $\mathbb{C}(t)$, and we have a rational map $Y\rightarrow X$ of degree 2 as the
dual isogeny.

{\bf{Notation}}
We consider  the fibration $\#n$ of $Y_k$ with a Weierstrass equation $E^{n}:y^{2}=x^{3}+A\left(
t\right)  x^{2}+B\left(  t\right)  x$ and
the two-torsion point $T=\left(  0,0\right).  $ We will call $\#n-i$ the elliptic
fibration $E^{n}/\langle \left(  0,0\right)  \rangle$ of the elliptic surface $Y_k/i$ if $i$ denotes the translation by $T$.

%(Koike)
%\subsubsection{Kummer surfaces}
\subsection{Fibrations of some Kummer surfaces}

Let $E_l$ be an  elliptic curve with invariant $j,$ defined by a Weierstrass equation in the Legendre form

\[E_l:y^2=x(x-1)(x-l).\]

Then  $l$ satisfies  the equation $j=256\frac{(1-l+l^2)^3}{l^2(l-1)^2}$. For a fixed $j$ the six values of $l$ are given by  $l$ or $\frac{1}{l}, 1-l,\frac{l-1}{l},\frac{-1}{l-1},\frac{l-1}{l}$.%l=-2\pm 2\sqrt{2}$ or $l=\frac{1\pm \sqrt{2}}{2}$.

Consider the Kummer surface $K$ given by $E_{l_1}\times E_{l_2}/\pm 1$ and choose as equation for $K$ 
\[x_1(x_1-1)(x_1-l_1)t^2=x_2(x_2-1)(x_2-l_2).\]

Following \cite {Ku} we can construct different elliptic fibrations. 
In the general case we can consider the three elliptic fibrations $F_i$ of $K$ defined by the elliptic parameters $m_i,$ with corresponding types of singular fibers
\[
\begin{array}
[c]{lll}
F_6:& m_6=\frac{x_1}{x_2}&2I_2^*,4I_2\\
F_8:&m_8=\frac{(x_2-l_2)(x_1-x_2)}{l_2(l_2-1)x_1(x_1-1)}&III^*,I_2^*,3I_2,I_1\\
%G_8:\, l_1=3+2\sqrt(2),l_2=l_1 & m_8 &III^*,I_3^*,3I_2 \\
F_5:& m_5=\frac{(x_1-x_2)(l_2(x_1-l_1)+(l_1-1)x_2)}{(l_2x_1-x_2)(x_1-l_1+(l_1-1)x_2)}&I_6^*,6I_2.\\
\end{array}
\]

In the special case when $E_1=E_2$ and $j_1=j_2=8000$ we obtain the following fibrations

\[
\begin{array}
[c]{lll}
F_6:\, l_1=l_2=3+2\sqrt(2)& m_6=\frac{x_1}{x_2}&2I_2^*,I_4,4I_2\\
F_8:\, l_1=3+2\sqrt(2),l_2=\frac{1}{l_1}&m_8=\frac{(x_2-l_2)(x_1-x_2)}{l_2(l_2-1)x_1(x_1-1)}&III^*,I_2^*,I_4,I_2,I_1\\
G_8:\, l_1=3+2\sqrt(2),l_2=l_1 & m_8 &III^*,I_3^*,3I_2 \\
F_5:\, l_1=l_2=3+2\sqrt(2)& m_5=\frac{(x_1-x_2)(l_2(x_1-l_1)+(l_1-1)x_2)}{(l_2x_1-x_2)(x_1-l_1+(l_1-1)x_2)}&I_6^*,I_4,4I_2.
\end{array}
\]

\subsection{Nikulin involutions and Kummer surfaces}

\begin{proposition}
\label{1} Consider a family $S_{a,b}$ of $K3$ surfaces with an elliptic
fibration, a two torsion section defining an involution $i$ and two
singular fibers of type $\ I_{4}^{\ast}$,%
\[
S_{a,b}:Y^{2}=X^{3}+\left(  t+\frac{1}{t}+a\right)  X^{2}+b^{2}X.
\]
Then the $K3$ surface $S_{a,b}/i$ is  the Kummer surface $\left(  E_{1}\times E_{2}\right)  /\left(  \pm Id\right)$
where the $j_i$ invariants of the elliptic curves $E_i$, $i=1,2$ are given by the formulae

\begin{align*}
j_{1}j_{2}  &  =4096\frac{\left(  a^{2}-3+12b^{2}\right)  ^{3}}{b^{2}}\\
\left(  j_{2}-1728\right)  \left(  j_{1}-1728\right)   &  =\frac
{1024a^{2}(2a^{2}-9-72b^{2})^{2}}{b^{2}}.%
\end{align*}
\end{proposition}

\begin{proof}

Recall that if  $E_{i}$, $i=1,2$, are two elliptic curves in the Legendre form
\[
E_{i}:y^{2}=x\left(  x-1\right)  \left(  x-l_{i}\right),
\]
the Kummer surface $K$ 
\[
K:\left(  E_{1}\times E_{2}\right)  /\left(  \pm Id\right)
\]
is defined by the following equation%
\[
x_{1}\left(  x_{1}-1\right)  \left(
x_{1}-l_{1}\right)  t^{2}=x_{2}\left(  x_{2}-1\right)  \left(  x_{2}%
-l_{2}\right).
\]

The Kummer surface $K$ admits an elliptic fibration with parameter
$u=m_6=\frac{x_{1}}{x_{2}}$ and Weierstrass equation
$H_{u}$
\[
H_{u}:Y^{2}=X\left(  X-u\left(  u-1\right)  (ul_{2}-l_{1})\right)  \left(
X-u\left(  u-l_{1}\right)  \left(  l_{2}u-1\right)  \right)  .
\]

The $2$-isogenous curve $S_{a,b}/\langle(0,0)\rangle$  has the following Weierstrass equation

\[
Y^{2}=X\left(  X-t\left(  t^{2}+\left(  a-2b\right)  t+1\right)  \right)
\left(  X-t\left(  t^{2}+\left(  a+2b\right)  t+1\right)  \right)  
\]
with two singular fibers of type $I_{2}^{\ast}$ above $0$ and $\infty.$

We easily prove that $S_{a,b}/\langle(0,0) \rangle$ and $H_{u}$ are isomorphic on the field
$\mathbb{Q}\left(  \sqrt{w_{2}}\right)  $ where%

\[
l_{1}=w_{1}^{\prime}w_{2}=\frac{w_{2}}{w_{1}}\text{ , }%\]
%\[
\quad  l_{2}=\frac{1}%
{w_{1}^{\prime}w_{2}^{\prime}}=w_{1}w_{2}\text{ and \ }t=w_{1}u,
\]
 $w_{1},w_{1}^{\prime}$ and $w_{2},w_{2}^{\prime}$ being respectively the roots
of polynomials ${t^{2}+\left(  a-2b\right)  t+1}$ and 
$t^{2}+\left(  a+2b\right)  t+1.$

Recall that the modular invariant $j_{i}$ of the elliptic curve $E_{i}$ is
linked to $l_{i}$ by the relation%
\[
j_{i}=256\frac{\left(  1-l_{i}+l_{i}^{2}\right)  ^{3}}{l_{i}^{2}\left(
1-l_{i}\right)  ^{2}}.%
\]
By elimination of $w_{1}$ and $w_{2},$ it follows the relations
between $j_{1}$ and $j_{2}$
\begin{align*}
j_{1}j_{2}  &  =4096\frac{\left(  a^{2}-3+12b^{2}\right)  ^{3}}{b^{2}}\\
\left(  j_{2}-1728\right)  \left(  j_{1}-1728\right)   &  =\frac
{1024a^{2}(2a^{2}-9-72b^{2})^{2}}{b^{2}}.%
\end{align*}

\end{proof}

%\subsection{Application}

In the Fermi family, the $K3$ surface $Y_k$
%\[
%X+\frac{1}{X}+Y+\frac{1}{Y}+Z+\frac{1}{Z}=k;\text{ with }k=s+\frac{1}{s}%
%\]
has the fibration $\#16$ with two singular fibers $I_{4}^{\ast},$ a
$2$-torsion point and Weierstrass equation
\[
y^{2}=x^{3}+x^{2}t\left(  4\left(  t^{2}+s^{2}\right)  +t\left(  s^{4}%
+14s^{2}+1\right)  \right)  +16t^{4}s^{6}x.
\]
Taking
\[
y=y^{\prime}t^{3}\left(  2\sqrt{s}\right)  ^{3},\quad x=x^{\prime}t^{2}\left(
2\sqrt{s}\right)  ^{2}\text{ and }t=t^{\prime}s,
\]
we obtain the following Weierstrass equation
\[
y^{\prime\acute{e}}=x^{^{\prime}3}+\left(  t^{\prime}+\frac{1}{t^{\prime}%
}+\frac{1}{4}\frac{s^{4}+14s^{2}+1}{s}\right)  +s^{4}x^{\prime}.%
\]
By the previous proposition with $a=\frac{1}{4}\frac{s^{4}+14s^{2}+1}%
{s},\quad b=s^{2},$ we derive the corollary below.

\begin{corollary} \label{3}
The surface obtained with the $2$-isogeny of kernel $\langle(0,0)\rangle $
from fibration \#16, is the Kummer surface associated to the product of two
elliptic curves of $j$-invariants $j_{1},j_{2}$  satifying%

\begin{align*}
j_{1}j_{2}  &  =\frac{\left(  s^{2}+1\right)  ^{3}\left(  s^{6}+219s^{4}%
-21s^{2}+1\right)  ^{3}}{s^{10}}\\
\left(  j_{1}-12^{3}\right)  \left(  j_{2}-12^{3}\right)   &  =\frac
{(s^{4}+14s^{2}+1)^{2}(s^{8}-548s^{6}+198s^{4}-44s^{2}+1)^{2}}{s^{10}}.%
\end{align*}
\end{corollary}

\begin{remark}
If $s=1$ we find $j_{1}=j_{2}=8000.$
\end{remark}

\begin{remark}
If $b=1$ we obtain the family of surfaces studied by Narumiya and
Shiga,  \cite{Na}. Morover if $a=\frac{9}{4}$ (resp. $4$) we find the two modular
surfaces associated to the modular groups $\Gamma_{1}(7)$ (resp. $\Gamma_{1}(8)).$
In these two cases we get $j_{1}=j_{2}=-3375$ (resp. $j_{1}=j_{2}=8000).$
\end{remark}

\begin{remark}
With the same method we can consider a family of $K3$ surfaces with Weierstrass equations
\[
E_{v}:Y^{2}+XY-(v+\frac{1}{v}-k)Y=X^{3}-(v+\frac{1}{v}-k)X^{2},%
\]
singular fibers  of type $2I_{1}^{\ast},2I_{4},2I_{1}$ and the point
$P_{v}=\left(  0,0\right)  $  of order $4.$ The elliptic curve
$E_{v}^{\prime}=E_{v}/\langle 2P_{v}\rangle$ has singular fibers of type $2I_{2}^{\ast
},4I_{2}.$ An analog computation gives $E_{v}^{\prime}\equiv\left(
E_{1}\times E_{2}\right)  /\left(  \pm Id\right)  $ and
\begin{align*}
j_{1}j_{2}  &  =(256k^{2}-16k-767)^{3}\\
\left(  j_{1}-12^{3}\right)  \left(  j_{2}-12^{3}\right)   &  =(32k-1)^{2}%
(128k^{2}-8k-577)^{2}.
\end{align*}
\end{remark}

\subsection{Shioda-Inose structure }

\begin{definition}
A $K3$ surface $X$ has a Shioda-Inose structure if there is a symplectic
involution $i$ on $X$ with rational quotient map $X\overset{p}{\rightarrow}Y$
such that $Y$ is a Kummer surface and $p^{\ast}$ induces a Hodge isometry
$T_{X}(2)\simeq T_{Y}$ .

Such an involution $i$ is called a Morrison-Nikulin involution. 
\end{definition}

An equivalent criterion is that $X$ admits a (Nikulin) involution
interchanging two orthogonal copies of $E_{8}(-1)$ in $NS(X)$, where $E_8(-1)$ is the unique
unimodular even negative-definite lattice of rank 8.

Or even more abstractly: $2E_{8}\hookrightarrow NS\left(  X\right)  $.

Applying this criterion to fibrations $\#17$ and $\#8$ and the Van Geemen-Sarti involution we get  the following result.

\begin{proposition}
The translation by the two torsion point of fibration $\#17$ and $\#8$ endowes  $Y_{k}$ with a Shioda-Inose structure.
\end{proposition}

Fibration $\#17$ has a fiber of type $I_{16}$ at $t=\infty.$ The idea \cite{G}
is to use the components $\Theta_{-2},\Theta_{-1},\Theta_{0}$, $\Theta_{1}$,  $\Theta_{2},\Theta_{3},\Theta_{4}$ of $I_{16}$ and the zero section to generate a
lattice of type $E_{8}.$ The two-torsion section intersects
$\Theta_{8}$ and the translation by the two-torsion point on the fiber $I_{16}$ transforms
$\Theta_{n}$ in $\Theta_{n+8}.$ The translation maps the lattice $E_{8}$
on an another disjoint $E_{8}$ lattice and defines a Shioda-Inose structure.

For fibration $\#8$,
the fiber above $t=0$ is of type $I_{6}$ and the section of order $2$
specialises to the singular point $\left(  0,0\right)  .$ Then after a blow up, 
it will not meet the $0$-component. If we denote $\Theta_{0,i}$,  $0\leq i\leq5$, 
the six components, then the zero section meets $\Theta_{0,0}$ and the
$2$-torsion section meets $\Theta_{0,3}.$ The translation by the $2$-torsion
section induces the permutation $\Theta_{0,i}\rightarrow\Theta_{0,i+3}.$

The fiber above $t=\infty$ is of type $I_{6}^{\ast}.$ The simple components
are denoted $\Theta_{\infty,0},\Theta_{\infty,1}$ and $\Theta_{\infty
,2},\Theta_{\infty,3}$; the double components are denoted $C_{i}$ with $0\leq
i\leq6$ and $\Theta_{\infty,0}.C_{0}=\Theta_{\infty,1}.C_{0}=1$;
$\Theta_{\infty,2}.C_{6}=\Theta_{\infty,3}.C_{6}=1.$ Then the $2$-torsion
section intersects $\Theta_{\infty,2}$ or $\Theta_{\infty,3}$ and the translation
by the $2$-torsion section induces the transposition $C_{i}\longleftrightarrow
C_{6-i}.$ 

The class of the components $C_{0},C_{1},C_{2},$  $\Theta_{\infty,0}%
,\Theta_{\infty,1},$ the zero section, $\Theta_{0,0}$ and  $\Theta_{0,1}$
define a $E_{8}\left(  -1\right)  .$ The Nikulin involution defined by the
two torsion section maps this  $E_{8}\left(  -1\right)  $ to another copy of
$E_{8}\left(  -1\right)  $  orthogonal to the first one; so the Nikulin
involution is a Morrison-Nikulin involution. 

\subsection{Base change and van Geemen-Sarti involutions} 
If a $K3$-surface $X$ has an elliptic fibration with two fibers of type $II^*$, this fibration can be realised by a Weierstrass equation of type
\[y^2=x^3-3\alpha x +(h+1/h-2\beta).\]
Moreover Shioda \cite{Shio} deduces the ``Kummer sandwiching'',  $\text{K} \rightarrow S \rightarrow \text{K}$, identifying the Kummer $\text{K}=E_1\times E_2/\pm 1$ with the help of the $j$-invariants of the two  elliptic curves $E_1,E_2$ and giving the following elliptic fibration of $\text{K}$
\[y^2=x^3-3\alpha x +(t^2+1/t^2-2\beta).\]
This can be viewed as a base change of the fibration of $X$.
\subsubsection{Alternate elliptic fibration} \label{2}
We shall now use an alternate elliptic fibration (\cite{Sc-Shio} example 13.6) to show that this construction is indeed a 2-isogeny between two elliptic fibrations of $S$ and $\text{K}.$ In the next picture we consider a divisor $D$ of type $I_{12}^*$ composed of the zero section $0$ and the components of the $II^*$ fibers enclosed in dashed lines. The far double components of the $II^*$ fibres can be chosen as sections of the new fibration. Take $\omega$ as the zero section. The other one is a two-torsion point since the function $h$ has a double pole on $\omega$ and a double zero at $M.$ It is the  function $'x'$ in a Weierstrass equation. More precisely with the new parameter $u=x$ and the variables  $Y=yh$ and $X=h,$ we obtain the Weierstrass equation  

\[  Y^2=X^3+(u^3-3\alpha u-2\beta)X^2+X .\]

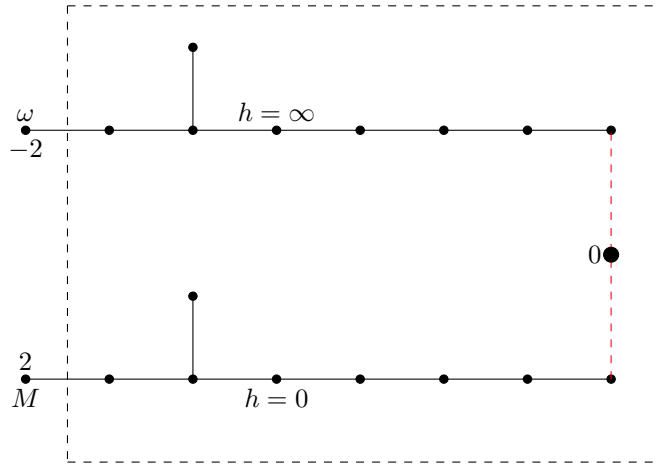
\begin{center}
\begin{tikzpicture}[scale=1.1]%
\draw[fill=black](0.1,1.1)circle (0.05cm)node [below]{$M$} node [above]{$2$};
\draw[fill=black](1.1,1.1)circle (0.05cm);%node [below] {$a_2$};
\draw[fill=black](2.1,1.1)circle (0.05cm);%node [below] {$a_3$};
\draw[fill=black](3.1,1.1)circle (0.05cm)node [below]{$h=0$};
\draw[fill=black](4.1,1.1)circle (0.05cm);
\draw[fill=black](5.1,1.1)circle (0.05cm);
\draw[fill=black](6.1,1.1)circle (0.05cm);
\draw[fill=black](7.1,1.1)circle (0.05cm);
\draw[fill=black](2.1,2.1)circle (0.05cm);
%\draw[fill=black](6.1,2.1)circle (0.05cm);
\draw(0.1,1.1)--(7.1,1.1);
%\draw(1.1,1.1)--(1.1,2.1);
\draw(2.1,1.1)--(2.1,2.1);

\draw(0.1,4.1)--(7.1,4.1);
\draw(2.1,4.1)--(2.1,5.1);
\draw[fill=black](0.1,4.1)circle (0.05cm)node [above]{$\omega$}node [below]{$-2$};
\draw[fill=black](1.1,4.1)circle (0.05cm);%node [below] {$a_2$};
\draw[fill=black](2.1,4.1)circle (0.05cm);%node [below] {$a_3$};
\draw[fill=black](3.1,4.1)circle (0.05cm)node [above]{$h=\infty$};
\draw[fill=black](4.1,4.1)circle (0.05cm);
\draw[fill=black](5.1,4.1)circle (0.05cm);
\draw[fill=black](6.1,4.1)circle (0.05cm);
\draw[fill=black](7.1,4.1)circle (0.05cm);
\draw[fill=black](2.1,5.1)circle (0.05cm);

\draw[dashed,red](7.1,1.1)--(7.1,4.1);
\draw[fill=black](7.1,2.6)circle (0.09cm)node [left]{$0$};
%\draw[dashed,red](7.1,2.6)--(9.1,2.6);
%\draw[dashed,red](7.1,2.6)--(9.1,4.1);
%\draw[fill=black](9.1,4.1)circle (0.05cm)node [above]{$\delta=-4/27$};;

%\draw(9.1,2.6)--(10.8,2.6);
%\draw(9.1,2.7)--(10.8,2.7);
%\draw[fill=black](9.1,2.6)circle (0.065cm)node [above]{$\delta=-2$};
%\draw[fill=black](10.8,2.6)circle (0.065cm);
\draw[dashed] (0.6,5.6)--(7.8,5.6);
\draw[dashed]  (0.6,5.6)--(0.6,0.1);
\draw[dashed] (7.8,5.6)--(7.8,0.1);
\draw[dashed] (0.6,0.1)--(7.8,0.1);

\end{tikzpicture}
\end{center}

In this equation, if we substitute $X(=h)$ by $t^2$, we obtain an equation in $W,\,t$ with $Y=Wt^2$, which is the equation for the 2-isogenous elliptic curve. 
Indeed  the birational transformation  
\[
%y = 4Y+4U^3+2UA, x = 2(Y+U^3)/U
y=4\,Y+4\,{U}^{3}+2\,UA,x=2\,{\frac {Y+{U}^{3}}{U}}
\]
with inverse 
\[
%U = 1/2y/(x+A),\qquad Y = 1/8(-y^2+2x^3+4x^2A+2xA^2)y/(x+A)^3
U=1/2\,{\frac {y}{x+A}},Y=1/8\,{\frac {\left (-{y}^{2}+2\,{x}^
{3}+4\,{x}^{2}A+2\,x{A}^{2}\right )y}{\left (x+A\right )^{3}}}%\right
\]

 transforms the curve $Y^2=U^6+A U^4+BU^2 $ in the Weierstrass form
\[
y^2=(x+A)(x^2-4B).
\]
This is an equation for the 2-isogenous curve of the curve $Y^2=X^3+AX^2+BX$ \cite{Si}.  
On the curve $Y^2=U^6+A U^4+BU^2 $, the involution $U \mapsto -U$ means adding the two-torsion point $(x=-A,y=0).$ 

Using this above process with $A=(u^3-3\alpha u-2\beta),$ the $2$-isogenous curve  $E_u$ has a Weierstrass equation 
\[
Y^2=(X+(u^3-3\alpha u-2\beta))(X^2-4)
\]

with singular fibers  of type $I_6^*,6I_2$. %Is I5 fibration?

The coefficients $\alpha$ and $\beta$ can be computed using the $j$-invariants

\[
\alpha^3=J_1J_2;  \,\,\, \beta^2=(1-J_1)(1-J_2);\,\,\,j_i=1728J_i.
\]
If  the elliptic curve is put in the Legendre form $y'^2=x'(x'-1)(x'-l)$ then $j=256\frac{(1-l+l^2)^3}{l^2(l-1)^2},$
 so  

\begin{align*}
\alpha^3=&\frac{16}{729}\frac{(1-l_1+l_1^2)^3(1-l_2+l_2^2)^3}{l_1^2(l_1-1)^2l_2^2(l_2-1)^2}\\
\beta=&\frac{1}{27}\frac{(2l_1-1)(l_1-2)(2l_2-1)(l_2-2)(l_1+1)(l_2+1)}{l_1l_2(l_1-1)(l_2-1)}.
\end{align*}

On the Kummer surface $E_{1}\times E_{2}/\pm1$ of equation
\[
X_{1}\left(  X_{1}-1\right)  \left(  X_{1}-l_{1}\right)  Z^2=X_{2}\left(
X_{2}-1\right)  \left(  X_{2}-l_{2}\right)
\]

we consider an elliptic fibration (case $\mathcal{J}_5$ of \cite{Ku}) with the parameter

$z=\frac{\left(  l_{2}X_{1}-X_{2}\right)  \left(
X_{1}-l_{1}+X_{2}\left(  l_{1}-1\right)  \right)  }{X_{2}\left(
X_{1}-1\right)  }$ (in fact $z=-\frac{l_1(l_2-1)}{m_5-1}$ cf. 4.2) and obtain the Weierstrass equation
\begin{align*}
\scriptstyle{ Y^{2} } &\scriptstyle{ =\left(  X-2l_{1}l_{2}\left(  l_{1}-1\right)  \left(  l_{2}-1\right)
\right)  \left(  X+2l_{1}l_{2}\left(  l_{1}-1\right)  \left(  l_{2}-1\right)
\right) } \\
& \scriptstyle{\left(  X+4z^{3}+4\left(  -2l_{1}l_{2}+l_{1}+l_{2}+1\right)  z^{2}+4\left(
l_{1}l_{2}-1\right)  \left(  l_{1}l_{2}-l_{1}-l_{2}\right)  z+2l_{1}%
l_{2}\left(  l_{1}-1\right)  \left(  l_{2}-1\right)  \right).}
\end{align*}
Substituting $z=w-\frac{1}{3}\left(  -2l_{1}l_{2}+l_{1}+l_{1}+1\right)  $
it follows
\begin{align*}
\scriptstyle{Y^{2}  }& =\scriptstyle{  \left(  X-2l_{1}l_{2}\left(  l_{1}-1\right)  \left(  l_{2}-1\right)
\right)  \left(  X+2l_{1}l_{2}\left(  l_{1}-1\right)  \left(  l_{2}-1\right)
\right)}  \\
&\scriptstyle{ \left(  X+4w^{3}-\frac{4}{3}\left(  l_{2}^{2}-l_{2}+1\right)  \left(
l_{1}^{2}-l_{1}+1\right)  w+ 
\frac{2}{27}\left(  l_{2}-2\right)  \left(
2l_{2}-1\right)  \left(  l_{1}-2\right)  \left(  2l_{1}-1\right)  \left(
l_{2}+1\right)  \left(  l_{1}+1\right)  \right).}
\end{align*}
Up to an automorphism of this Weierstrass form we recover the equation of $E_{u}.$

The previous results can be used to show the following proposition

\begin{proposition}
The translation by the two torsion point of the elliptic fibration $\#4$ gives to $Y_{k}$ a
Shioda-Inose structure.
\end{proposition}

\section{Proof of Theorem $1.1$ }

%\begin{notation}
%We consider  the fibration \#n of $Y_k$ with a Weierstrass equation $E_{n}:y^{2}=x^{3}+A\left(
%t\right)  x^{2}+B\left(  t\right)  x$ for the generic elliptic curve with
%the two-torsion point $\tau=\left(  0,0\right).  $ We will call $\#n-i$ the elliptic
%fibration $E_{n}/<\left(  0,0\right)  >$ of the elliptic surface $Y_k/i_\tau$ if $i_\tau$ is the translation by $\tau$. 
%\end{notation}

We consider an elliptic fibration $\#n$ of $Y_{k}$ with a two torsion section.

From the Shioda-Tate formula (cf. e.g. \cite{Shio1}, Corollary 1.7]) we have the
relation
\[
12=\frac{|\Delta|\prod m_{v}^{\left(  1\right)  }}{|\text{Tor}|^{2}}%
\]
where $\Delta$ is the determinant of the height-matrix of a set of generators
of the Mordell-Weil group, $m_{v}^{\left(  1\right)  }$ the number of simple
components of a singular fiber and $|\text{Tor}|$ the order of the torsion group of
Mordell-Weil group. This formula allows us to determine generators of the
Mordell-Weil group except for fibration $\#4$. Using the $2$-isogeny we  determine also the Mordell-Weil group of $\#n$-i. The discriminant is either
$12\times2$ or $12\times8.$

\begin{proposition}
The translation by the two torsion point of the fibration $\#16$ gives to $Y_{k}$ a
Shioda-Inose structure.
\end{proposition}

From the previous Proposition \ref{1}, the translation by the two torsion
point of $\#16$ gives to the quotient a Kummer structure. The fibration $\#16$ is
of rank one, its Mordell-Weil group is generated by $P$ and the two torsion point. By
computation we can see that the Mordell-Weil group of the $2$-isogenous curve on
$E(\mathbb{C}\left(  t\right))$  is generated by $p\left(  P\right)  $ and
torsion sections. So we can compute the discriminant of the N\'{e}ron-Severi
group which is $12\times8$. The second condition, $T_{X}(2)\simeq T_{Y},$ is
then verified.

\begin{remark}
The $K3$ surface of Picard number $20$ given with the elliptic fibration
\[
Y^{2}=X^{3}-\left(  t+\frac{1}{t}-\frac{3}{2}\right)  X^{2}+\frac{1}{16}X
\]
or 
\[
{y}^{2}={x}^{3}-1/2\,t\left (2\,{t}^{2}+2-3\,t\right ){x}^{2}+1/16\,{t
}^{4}x
\]
has rank $1.$ The Mordell-Weil group is generated by $\left(  0,0\right)  $
and $P=(x=\frac{1}{4},y=\frac{\left(  t-1\right)  ^{2}}{8}).$ The determinant
of the N\'{e}ron-Severi group is equal to $12.$ By computation we have $p\left(
P\right)  =2Q$ with $Q=\left(  t(t-1)(t^{2}-t+1),-t^{3}(t-1)(t^{2}%
-t+1)\right)  $ of height $\frac{3}{4}.$ The determinant of the N\'{e}ron Severi
group of the $2$-isogenous curve is then $12$ not $12\times2^{2}.$ So the
involution induced by the two-torsion point is not a Nikulin-Morrison involution.
Moreover the $2$-isogenous elliptic curve is a fibration of the Kummer surface $E\times E/\pm 1$ where $j(E)=0.$  
\end{remark}

For fibrations $\#n$-i with discriminant of the transcendental lattice $12\times8$ we prove that we have the Shioda-Inose structure in
the following way: from corollary 4.1 this is true for $\#16$-i, from Proposition 4.3 this is true for $\#4$-i and from Proposition 4.2 this is true for
$\#17$-i, $\#8$-i. The other fibrations $\#n$-i can be obtained by $2$- or $3$-%
neighbor method from $ \#16$-i, $\#8$-i or $\#17$-i. The results are given
in the Table \ref{Ta:Iso8}. In the second column are written the  Weierstrass equations for the $\#n$ elliptic fibration 
and its $2$-isogenous fibration, singular fibers and the $x$-coordinates of generators of the Mordell lattice of $\#n$-i. In the third column we give the starting fibration for the $2$- or
$3$- neighbor method and in the last column the parameter used from the starting fibration.%
\subsection{The $K3$ surface $S_k$}
For the remaining fibrations, (discriminant $12\times2),$ using also the $2$- or $3$-
neighbor method, they are proved to lie on the same surface $S_k.$ Except for
the case $\#7$ the results are collected in the Table \ref{Ta:Iso2} with the same format. The
case $\#7$ needs an intermediate fibration explained in the next paragraph.%

Starting with fibration $\#7$-i and using the parameter $m_{7}=\frac
{y}{xt\left(  t-s^{2}\right)  }$ it follows the Weierstrass equation
\begin{align*}
Y^{2}+2\left(  m_{7}^{2}s^{2}-2\right)  YX-16m_{7}^{4}s^{4}Y =\left(
X-8m_{7}^{2}s^{2}\right) & \\  \left(  X+8m_{7}^{2}s^{2}\right)  \left(  X+m_{7}^{2}\left(
s^{4}-6s^{2}+1\right)  -4\right) & \\
\end{align*}
with singular fibers  $I_{8}\left(  \infty\right)  ,IV^{\ast}\left(  0\right)  ,8I_{1} $. 

Then the parameter $m_{15}=\frac{Y}{\left(  X+8m_{7}^{2}s^{2}\right)  }$ leads to the fibration $\#15$-i.

\begin{table}[tp]\footnotesize
\[%
\begin{array}
[c]{|c|c|c|c|}

%\begin{tabular}
\hline
\text{No} & \text{Weierstrass Equation} & \text{From} & \text{Param.}\\\hline
\#4 & \text{see Prop 9} &  & \\\hline\hline
\#8 &
\begin{array}[c]{c}%
\begin{array}
[c]{c}%
y^{2}-k\left(  t-1\right)  yx=x\left(  x-1\right)  \left(  x-t^{3}\right)
\\\hline
I_{6}^{\ast}\left(  \infty\right)  ,I_{6}\left(  0\right)  ,I_{2}\left(
1\right)  ,4I_{1}%
\end{array}
\\\hline
y^{2}=x^{3}+\frac{1}{2}\left(  4t^{3}-t^{2}k^{2}+2tk^{2}+4-k^{2}\right)
x^{2}\\
+\frac{1}{16}\left(  t-1\right)  ^{2}\left(  4t^{2}+t\left(  4-k^{2}\right)
+\left(  k-2\right)  ^{2}\right)  \\
\left(  4t^{2}+t\left(  4-k^{2}\right)  +\left(  k+2\right)  ^{2}\right)
x\\\hline
I_{3}^{\ast}\left(  \infty\right)  ,I_{3}\left(  0\right)  ,I_{4}\left(
1\right)  ,4I_{2}\\\hline
x_{Q}=-\frac{1}{4}\left(  t-1\right)  \left(  4t^{2}+t\left(  4-k^{2}\right)
+\left(  k-2\right)  ^{2}\right)
\end{array}
&  & \\\hline\hline
\#16 &
\begin{array}
[c]{c}%
y^{2}=x^{3}+t\left(  4\left(  t^{2}+s^{2}\right)  +t\left(  s^{4}%
+14s^{2}+1\right)  \right)  x^{2}+16s^{6}t^{4}x\\\hline
2I_{4}^{\ast}\left(  \infty,0\right)  ,4I_{1}\\\hline
y^{2}=x\left(  x-t\left(  4\left(  t^{2}+s^{2}\right)  +t\left(  s^{4}%
+14s^{2}+1+8s^{3}\right)  \right)  \right)  \\
\left(  x-t\left(  4\left(  t^{2}+s^{2}\right)  +t\left(  s^{4}+14s^{2}%
+1-8s^{3}\right)  \right)  \right)  \\\hline
2I_{2}^{\ast}\left(  \infty,0\right)  ,4I_{2}\\\hline
x_{Q}=\frac{t^{2}\left(  \left(  t^{2}+s^{2}\right)  \left(  3+s^{2}\right)
+t\left(  -s^{4}+8s^{2}+1\right)  \right)  ^{2}}{\left(  t+1\right)
^{2}\left(  t+s^{2}\right)  ^{2}}%
\end{array}
&  & \\\hline\hline
\#17 &
\begin{array}
[c]{c}%
\begin{array}
[c]{c}%
y^{2}-\frac{1}{2}\left(  s^{4}+14s^{2}+1-s^{2}t^{2}\right)  yx=x\left(
x-4s^{2}\right)  \left(  x-4s^{6}\right)  \\\hline
I_{16}\left(  \infty\right)  ,8I_{1}\left(  \pm\frac{s^{2}\pm4s-1}%
{s},..\right)
\end{array}
\\\hline
y^{2}=x\left(  x-\left(  t^{2}s^{2}-\left(  s^{4}+14s^{2}+1\right)
\pm8s\left(  s^{2}+1\right)  \right)  \right)  \\
\left(  x-\left(  ts+s^{2}\pm4s-1\right)  \left(  ts-s^{2}\pm4s+1\right)
\right)  \\\hline
I_{8}\left(  \infty\right)  ,8I_{2}\left(  \pm\frac{s^{2}\pm4s-1}%
{s},..\right)  \\\hline
x_{Q_{1}}=\frac{1}{16}\left(  ts+s^{2}-4s-1\right)  \left(  ts-s^{2}%
+4s+1\right)  \\
\left(  t^{2}s^{2}-\left(  s^{4}+14s^{2}+1\right)  +8s\left(  s^{2}+1\right)
\right)  \\
x_{Q_{2}=}\frac{1}{16}\frac{\left(  s-1\right)  ^{2}}{\left(  s+1\right)
^{2}}\left(  ts+s^{2}+4s-1\right)  \\
\left(  ts-s^{2}+4s+1\right)  \left(  t^{2}s^{2}-\left(  s^{4}+14s^{2}%
+1\right)  -8s\left(  s^{2}+1\right)  \right)
\end{array}&  & \\\hline\hline
\#23 &
\begin{array}
[c]{c}%
\begin{array}
[c]{c}%
y^{2}+\left(  2t^{2}-tk+1\right)  yx=x\left(  x-t^{2}\right)  \left(
x-t^{4}\right)  \\\hline
I_{0}^{\ast}\left(  \infty\right)  ,I_{12}\left(  0\right)  ,6I_{1}\left(
\frac{1}{k\pm2},..\right)
\end{array}
\\\hline 
y^{2}=x\left(  x+\frac{1}{4}\left(  t(k-2)-1\right)  \left(  4t^{2}-\left(
k+2\right)  t+1\right)  \right)  \\
\left(  x+\frac{1}{4}\left(  t\left(  k+2\right)  -1\right)  \left(
4t^{2}-\left(  k-2\right)  t+1\right)  \right)  \\\hline
I_{0}^{\ast}\left(  \infty\right)  ,I_{6}\left(  0\right)  ,6I_{2}\left(
\frac{1}{k\pm2},..\right)  \\\hline%
\begin{array}
[c]{c}%
x_{Q_{1}}=\frac{-1}{4}\left(  4t^{2}-\left(  k-2\right)  t+1\right)  \left(
t\left(  k-2\right)  -1\right)  ;\\
x_{Q_{2}=}\frac{-1}{4}\frac{k-2}{k+2}\left(  t\left(  k+2\right)  -1\right)
\left(  4t^{2}-\left(  k-2\right)  t+1\right)
\end{array}
\end{array}
&
\begin{array}
[c]{c}%
\#8\\
\#26\\
\#24
\end{array}
&
\begin{array}
[c]{c}%
\frac{y-y_{2Q}+\frac{k}{2}\left(  x-x_{2Q}\right)  }{t\left(  x-x_{2Q}\right)
}\\
\frac{2y-\left(  t-s\right)  \left(  ts-1\right)  x}{t\left(  x-ts\left(
ts-1\right)  ^{2}\right)  }\\
\frac{\left(  y-y_{2Q_{2}}\right)  +\frac{s^{2}+1}{2}\left(  x-x_{2Q_{2}%
}\right)  }{\left(  t+1\right)  \left(  x-x_{2Q_{2}}\right)  }%
\end{array}
\\\hline\hline
\#24 &
\begin{array}
[c]{c}%
\begin{array}
[c]{c}%
y^{2}+\left(  s^{2}+1\right)  tyx=x\left(  x-t^{2}s^{2}\right)  \left(
x-s^{2}t\left(  t+1\right)  ^{2}\right)  \\\hline
2I_{2}^{\ast}\left(  \infty,0\right)  ,I_{4}\left(  -1\right)  ,4I_{1}%
\end{array}
\\\hline
y^{2}=x^{3}+\frac{1}{2}t\left(  4t^{2}s^{2}-t\left(  s^{4}-10s^{2}+1\right)
+4s^{2}\right)  x^{2}\\
+\frac{1}{16}t^{2}\left(  4t^{2}s^{2}+\left(  8s^{2}-\left(  s-1\right)
^{4}\right)  t+4s^{2}\right)  \\
\left(  4t^{2}s^{2}+\left(  8s^{2}-\left(  s+1\right)  ^{4}\right)
t+4s^{2}\right)  x\\\hline
2I_{1}^{\ast}\left(  \infty,0\right)  ,5I_{2}\left(  -1,\right)  \\\hline
x_{Q_{1}}=\frac{1}{4}(2t^{2}s^{2}+t\left(  s^{2}-1\right)  -2)^{2};\\
x_{Q_{2}}=-\frac{1}{4}t\left(  4t^{2}s^{2}+\left(  8s^{2}-\left(  s+1\right)
^{4}\right)  t+4s^{2}\right)
\end{array}
&  & \\\hline\hline
\#26 &
\begin{array}
[c]{c}%
y^{2}+(ts-1)(t-s)xy=x\left(  x-t^{2}s^{2}\right)  ^{2}\\\hline
2I_{8}\left(  \infty,0\right)  ,I_{2}\left(  s,\frac{1}{s}\right)
,4I_{1}\\\hline
y^{2}=x\left(  x+4t^{2}s^{2}\right)  \left(  x+\frac{1}{4}\left(  t-s\right)
^{2}\left(  st-1\right)  ^{2}\right)  \\\hline
I_{4}\left(  \infty,0,s,\frac{1}{s}\right)  ,4I_{2}\\\hline
x_{Q_{1}}=ts\left(  ts-1\right)  ^{2}%
\end{array}
&  & \\\hline
\end{array}
\]
\caption{ Fibrations with discriminant $12\times 8$ (Fibrations of the Kummer $K_k$)} \label{Ta:Iso8}
\end{table}

%\[%
\begin{table}[tp]\footnotesize

\begin{center}
\begin{tabular}
[c]{|c|c|c|c|}\hline
\text{No} & \text{Weierstrass Equation} & \text{From} & \text{Param.}\\\hline
\#7 &
$\begin{array}
[c]{c}%
\begin{array}
[c]{c}%
y^{2}=x^{3}+\frac{1}{4}t\left(  t\left(  s^{4}-10s^{2}+1\right)
+8s^{4}\right)  x^{2}-t^{2}s^{2}\left(  t-s^{2}\right)  ^{3}x\\\hline
III^{\ast}\left(  \infty\right)  ,I_{1}^{\ast}\left(  0\right)  ,I_{6}\left(
s^{2}\right)  ,2I_{1}%
\end{array}
\\\hline
y^{2}=x^{3}-\frac{1}{2}t\left(  t\left(  s^{4}-10s^{2}+1\right)
+8s^{4}\right)  x^{2}\\
+\frac{1}{16}t^{3}\left(  64t^{2}s^{2}+\left(  s^{8}-20s^{6}-90s^{4}%
-20s^{2}+1\right)  t+16s^{4}\left(  s^{2}+1\right)  ^{2}\right)  x\\\hline
III^{\ast}\left(  \infty\right)  ,I_{2}^{\ast}\left(  0\right)  ,I_{3}\left(
s^{2}\right)  ,2I_{2}\\%\hline
\end{array}$
&  & \\\hline\hline
\#9 &
$\begin{array}
[c]{c}%
\begin{array}
[c]{c}%
y^{2}=x^{3}+\frac{1}{4}\left(  s^{4}+14s^{2}+1\right)  t^{2}x^{2}+t^{3}%
s^{2}\left(  s^{2}+t\right)  \left(  ts^{2}+1\right)  x\\\hline
2III^{\ast}\left(  \infty,0\right)  ,2I_{2}\left(  -s^{2},-\frac{1}{s^{2}%
}\right)  ,2I_{1}\\\hline
\end{array}
\\
y^{2}=x^{3}-\frac{1}{2}\left(  s^{4}+14s^{2}+1\right)  t^{2}x^{2}\\
-\frac{1}{16}t^{3}\left(  64t^{2}s^{4}+t\left(  -s^{8}+36s^{6}-198s^{4}%
+36s^{2}-1\right)  +64s^{4}\right)  x\\\hline
2III^{\ast}\left(  \infty,0\right)  ,2I_{1}\left(  -s^{2},-\frac{1}{s^{2}%
}\right)  ,2I_{2}\\\hline
x_{Q}=\frac{1}{4}\frac{\left(  t+1\right)  ^{2}\left(  2t^{2}s^{2}+t\left(
s^{4}-6s^{2}+1\right)  +2s^{2}\right)^2  }{\left(  s^{2}-1\right)  ^{2}\left(
t-1\right)  ^{2}}%
\end{array}$
& \#20 &$ \frac{y}{\left(  ts-1\right)  ^{4}}$\\\hline\hline
\#14 &
$\begin{array}
[c]{c}%
y^{2}=x^{3}+\left(  t^{3}\left(  s^{4}+1\right)  +\frac{1}{4}t^{2}\left(
s^{4}+14s^{2}+1\right)  +ts^{2}\right)  x^{2}+t^{6}s^{4}x\\\hline
I_{0}^{\ast}\left(  \infty\right)  ,I_{8}^{\ast}\left(  0\right)
,4I_{1}\left(  \frac{-1}{4},\frac{-4s^{2}}{\left(  s^{2}-1\right)  ^{2}%
},..\right)  \\\hline
y^{2}=x\left(  x-\frac{1}{4}\left(  s^{2}+1\right)  ^{2}t^{3}-\frac{1}%
{4}\left(  s^{4}+14s^{2}+1\right)  t^{2}-ts^{2}\right)  \\
\left(  x-\frac{1}{4}\left(  s^{2}-1\right)  ^{2}t^{3}-\frac{1}{4}\left(
s^{4}+14s^{2}+1\right)  t^{2}-ts^{2}\right)  \\\hline
I_{0}^{\ast}\left(  \infty\right)  ,I_{4}^{\ast}\left(  0\right)
,4I_{2}\left(  \frac{-1}{4},\frac{-4s^{2}}{\left(  s^{2}-1\right)  ^{2}%
},..\right)  \\\hline
x_{Q}=\frac{1}{4}t^{2}\left(  s^{2}-1\right)  ^{2}\left(  4t+1\right)
\end{array}$
& \#15 & $\frac{t^{2}s^{2}}{x+t^{3}s^{2}-\frac{1}{4}t^{2}\left(  s^{2}%
-1\right)  ^{2}}$\\\hline\hline
\#15 &
$\begin{array}
[c]{c}%
\begin{array}
[c]{c}%
\left(  y-tx\right)  \left(  y-s^{2}tx\right)  =x\left(  x-ts^{2}\right)
\left(  x-ts^{2}\left(  t+1\right)  ^{2}\right)  \\\hline
I_{4}^{\ast}\left(  \infty\right)  ,I_{1}^{\ast}\left(  0\right)
,I_{4}\left(  -1\right)  ,3I_{1}\left(  \frac{1}{4}\left(  \frac{s^{2}-1}%
{s}\right)  ^{2},..\right)
\end{array}
\\\hline
y^{2}=x\left(  x+t^{3}s^{2}-\frac{1}{4}t^{2}(s^{2}-1)^{2}\right)  \\
\left(  x+t^{3}s^{2}-\frac{1}{4}(s^{2}-4s-1)(s^{2}+4s-1)t^{2}+4ts^{2}\right)
\\\hline
I_{2}^{\ast}\left(  \infty,0\right)  ,4I_{2}\left(  -1,\frac{1}{4}\left(
\frac{s^{2}-1}{s}\right)  ^{2},..\right)  \\\hline
x_{Q}=\frac{1}{4}t^{2}(s^{2}-1)^{2}%
\end{array}$
& \#20 & $\frac{y}{x\left(  t-s\right)  ^{2}}$\\\hline\hline
\#20 &
$\begin{array}
[c]{c}%
\begin{array}
[c]{c}%
y^{2}-(t^{2}s-\left(  s^{2}+1\right)  t+3s)yx-s^{2}\left(  t-s\right)  \left(
ts-1\right)  y=x^{3}\\\hline
I_{12}\left(  \infty\right)  ,2I_{3}\left(  s,\frac{1}{s}\right)
,2I_{2}\left(  0,\frac{s^{2}+1}{s}\right)  ,2I_{1}%
\end{array}
\\\hline
y^{2}+\left(  t^{2}s-(s^{2}+1)t-3s\right)  yx-s\left(  t-s\right)  ^{2}\left(
ts-1\right)  ^{2}y=x^{3}\\\hline
3I_{6}\left(  \infty,s,\frac{1}{s}\right)  ,2I_{2},2I_{1}\left(  0,\frac
{s^{2}+1}{s}\right)
\end{array}$
&  & \\\hline
\end{tabular}
\end{center}
\caption{ Fibrations with discriminant $12\times 2$ (Fibrations of $S_k$)} \label{Ta:Iso2}
\end{table}

%\end{array}
%\]

%\bigskip

For the last part of Theorem 1.1 we give properties of $S_k$. First we prove that $S_k$ is the Jacobian variety of some genus $1$ fibrations of $K_k$.

Starting with the fibration $\#26$-i and Weierstrass equation%
\[
y^{2}=x\left(  x+4t^{2}s^{2}\right)  \left(  x+\frac{1}{4}\left(  t-s\right)
^{2}\left(  ts-1\right)  ^{2}\right)
\]
the new parameter $m:=\frac{y}{t\left(  x+\frac{1}{4}\left(  t-s\right)
^{2}\left(  ts-1\right)  ^{2}\right)  }$ defines an elliptic fibration of
$\#26$-i with Weierstrass equation
\[
E_{m}:Y^{2}-m\left(  s^{2}+1\right)  YX=X\left(  X-s^{2}m^{2}\right)  \left(
X+\frac{1}{4}\left(  2m-s\right)  ^{2}\left(  2m+s\right)  ^{2}\right)
\]
and singular fibers are of type $4I_{4}\left(  0,\pm\frac{1}{2}s,\infty\right),8I_{1}.$

Then setting as new parameter $n=\frac{X}{m^{2}}$, it follows a genus one
curve in $m$ and $Y.$ Its equation, of degree 2 in $Y$, can be transformed
in
\[
w^{2}=-16n(-n+s^{2})m^{4}+n(s^{4}\left(  8+n\right)  -10ns^{2}+n\left(
1+4n\right)  )m^{2}-ns^{4}(-n+s^{2}).
\]
Let us recall the formulae giving the jacobian of a genus one curve defined by the equation
$y^{2}=ax^{4}+bx^{3}+cx^{2}+dx+e$. If $c_{4}=2^{4}(12ae-3bd+c^{2})$ and
$c_{6}=2^{5}(72ace-27ad^{2}-27b^{2}e+9bcd-2c^{3})$, then the equation of the
Jacobian curve is
\[
\bar{y}^{2}=\bar{x}^{3}-27c_{4}\bar{x}-54c_{6}.%
\]
In our case we obtain
\[
y^{2}  =x\left(  x+n^{3}s^{2}-\frac{1}{4}n^{2}\left(  s^{2}-1\right)^{2}\right)  \]
\[\left(  x+n^{3}s^{2}-\frac{1}{4}\left(  s^{2}-4s-1\right)  \left(s^{2}+4s-1\right)  n^{2}+4ns^{2}\right),\]

%\[
%y^{2}=x\left(  x+n^{3}s^{2}-\frac{1}{4}n^{2}\left(  s^{2}-1\right)
%^{2}\right)  \left(  x+n^{3}s^{2}-\frac{1}{4}\left(  s^{2}-4s-1\right)
%\left(  s^{2}+4s-1\right)  n^{2}+4ns^{2}\right)
%\]
which is precisely the fibration $\#15$-i.

\begin{remark}
%Another result can be derived from $E_{m}$  using
%the new parameter
Using the new parameter $p=\frac{Y}{m^{2}\left(  X+\frac{1}{4}\left(  2m-s\right)
^{2}\left(  2m+s\right)  ^{2}\right)  }$ another result can be derived from $E_m$ leading to
%\[
%E_{p}:Y^{2}-2s\left(  2p-1\right)  \left(  2p+1\right)  YX-X\left(
%X+64s^{2}p^{2}\right)  \left(  X+\left(  2sp+1\right)  \left(  2sp-1\right)
%\left(  s+2p\right)  \left(  s-2p\right)  \right)
%\]
\begin{align*}
E_{p}  & :Y^{2}-2s\left(  2p-1\right)  \left(  2p+1\right)  YX=\\
& X\left(  X+64s^{2}p^{2}\right)  \left(  X+\left(  2sp+1\right)  \left(
2sp-1\right)  \left(  s+2p\right)  \left(  s-2p\right)  \right),
\end{align*}
with singular fibers $2I_{0}^{\ast},4I_{2},4I_{1}.$ From $E_{p}$ and the
new parameter $k=\frac{X}{p^{2}}$ we obtain a genus one fibration whose
jacobian is $\#14$-i. 

Starting from the fibration $\#26-i$, the parameter $q=\frac{x}{t^{2}}$
leads to a genus one fibration  whose jacobian is the fibration
$ \#20$-i. 
\end{remark}

\subsection{Transcendental and N\'eron-Severi lattices of the surface $S_k$}

\begin{lemma}
The five fibrations $\#7-i$, $\#9-i$, $\#14-i$, $\#15-i$, $\#20-i$ are fibrations of the same $K3$ surface $S_k$ with transcendental lattice
\[
T_{S_k}=\langle (-2)\rangle\oplus \langle 2\rangle \oplus \langle 6 \rangle
\]
and N\'eron-Severi lattice 
\[NS(S_k)=U\oplus E_8(-1) \oplus E_7(-1) \oplus \langle (-2)\rangle \oplus \langle (-6) \rangle .\]
Moreover these fibrations specialise in fibrations of $Y_2$ for $k=2$.
\end{lemma}

\begin{proof}
These five fibrations are respectively the fibrations given in Table 6 and recalled below with the type of their singular fibers, their rank and torsion group:
\[
\begin{matrix}
\#7-i & 2A_1A_2D_6E_7 & \text{rk} \,\,0 & \mathbb Z/2\mathbb Z\\
\#9-i & 2A_12E_7 & \text{rk}\,\, 1 &  \mathbb Z/2\mathbb Z\\
\#14-i & 4A_1D_4D_8 & \text{rk}\,\, 1 &  \mathbb Z/2\mathbb Z \times  \mathbb Z/2\mathbb Z \\
\#15-i & 4A_12D_6 & \text{rk}\,\, 1 &  \mathbb Z/2\mathbb Z \times  \mathbb Z/2\mathbb Z \\
\#20-i & 2A_1 3A_5 & \text{rk}\,\, 0 &  \mathbb Z/6\mathbb Z.
\end{matrix}
\]
We already know from the above results that they are fibrations of $S_k$ but we give below another proof of this fact.
Denote $S_k$ the $K3$ surface defined by the elliptic fibration $\#7$-i with Weierstrass equation given in Table 6 and draw the graph of the singular fibers, the zero and two-torsion sections of the elliptic fibration

%\[
\begin{align*}
Y^2={X}^{3}+\left (\left (-1/2\,{s}^{4}+5\,{s}^{2}-1/2\right ){t}^{2}-4\,{
s}^{4}t\right ){X}^{2}\\%\]
 \left (4\,{s}^{2}{t}^{5}+\left (1/16\,{s}^{8}-{\frac {45}{8}}\,{s}^{4}
-5/4\,{s}^{2}-5/4\,{s}^{6}+1/16\right ){t}^{4}+{s}^{4}\left ({s}^{2}+1
\right )^{2}{t}^{3}\right )X \\%
\end{align*}
%\]
with singular fibers $III^*(\infty),I^*_2(0),I_3(s^2),2I_1(t_1,t_2)$.

\bigskip

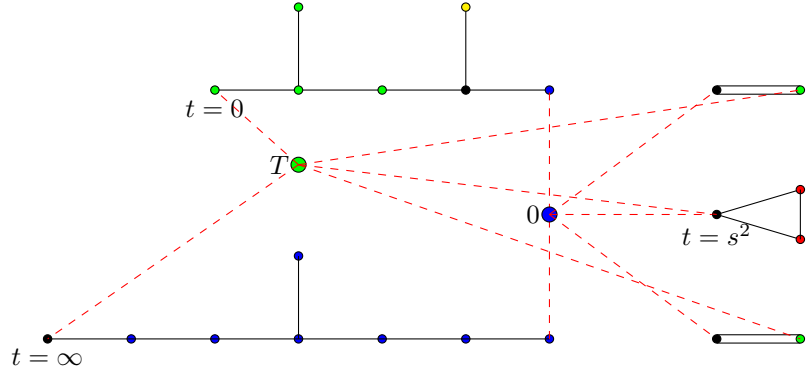
\begin{center}
\begin{tikzpicture}[scale=1.1]%
%\draw[fill=black](0.1,1.1)circle (0.05cm)node [below]{$M$} node [above]{$2$};
\draw[fill=black](1.1,1.1)circle (0.05cm)node [below] {$t=\infty$};
\draw[fill=blue](2.1,1.1)circle (0.05cm);% node [below] {$e_6$};
\draw[fill=blue](3.1,1.1)circle (0.05cm);% node [below]{$e_5$};
\draw[fill=blue](4.1,1.1)circle (0.05cm);% node [below]{$e_4$};
\draw[fill=blue](5.1,1.1)circle (0.05cm);%node [below]{$e_3$};
\draw[fill=blue](6.1,1.1)circle (0.05cm);%node [below]{$e_1$};
\draw[fill=blue](7.1,1.1)circle (0.05cm);%[below]{$e_1$};
\draw[fill=blue](4.1,2.1)circle (0.05cm) ;%node [right] {$e_2$};
%\draw[fill=black](6.1,2.1)circle (0.05cm);
\draw(1.1,1.1)--(7.1,1.1);
%\draw(1.1,1.1)--(1.1,2.1);
\draw(4.1,1.1)--(4.1,2.1);

\draw(3.1,4.1)--(7.1,4.1);
\draw(4.1,4.1)--(4.1,5.1);
\draw(6.1,4.1)--(6.1,5.1);
%\draw[fill=black](0.1,4.1)circle (0.05cm);%node [above]{$\omega$}node [below]{$-2$};
%\draw[fill=black](1.1,4.1)circle (0.05cm);%node [below] {$a_2$};
%\draw[fill=black](2.1,4.1)circle (0.05cm);%node [below] {$a_3$};
\draw[fill=green](3.1,4.1)circle (0.05cm)node [below]{$t=0$};%node [above]{$h=\infty$};
\draw[fill=green](4.1,4.1)circle (0.05cm);%node [below]{$d_4$};
\draw[fill=green](5.1,4.1)circle (0.05cm);%node [below]{$d_3$};
\draw[fill=black](6.1,4.1)circle (0.05cm);%node [below]{$d_2$};
\draw[fill=blue](7.1,4.1)circle (0.05cm);
\draw[fill=green](4.1,5.1)circle (0.05cm);%node [right]{$d_6$};
\draw[fill=yellow](6.1,5.1)circle (0.05cm);%node [right]{$d_1$};

\draw[dashed,red](7.1,1.1)--(7.1,4.1);
\draw[fill=blue](7.1,2.6)circle (0.09cm)node [left]{$0$};
\draw[dashed,red](7.1,2.6)--(9.1,2.6);
\draw[dashed,red](7.1,2.6)--(9.1,4.1);
\draw[dashed,red](7.1,2.6)--(9.1,1.1);
\draw[fill=black](9.1,4.1)circle (0.05cm);%node [above]{$\delta=-4/27$};;
\draw(9.1,4.15)--(10.1,4.15);
\draw(9.1,4.05)--(10.1,4.05);
\draw[fill=green](10.1,4.1)circle (0.05cm);
\draw[fill=black](9.1,2.6)circle (0.05cm) node [below]{$t=s^2$};;
\draw[fill=black](9.1,1.1)circle (0.05cm);
\draw(9.1,1.15)--(10.1,1.15);
\draw(9.1,1.05)--(10.1,1.05);
\draw[fill=green](10.1,1.1)circle (0.05cm);

\draw(9.1,2.6)--(10.1,2.9);
\draw(9.1,2.6)--(10.1,2.3);
\draw[fill=red](10.1,2.9)circle (0.05cm);
\draw[fill=red](10.1,2.3)circle (0.05cm);
\draw(10.1,2.3)--(10.1,2.9);

\draw[fill=green](4.1,3.2)circle (0.09cm) node[left]{$T$};%n
\draw[dashed,red](4.1,3.2)--(1.1,1.1);
\draw[dashed,red](4.1,3.2)--(3.1,4.1);
\draw[dashed,red](4.1,3.2)--(9.1,2.6);
\draw[dashed,red](4.1,3.2)--(10.1,1.1);
\draw[dashed,red](4.1,3.2)--(10.1,4.1);

%\draw(9.1,2.6)--(10.8,2.6);
%\draw(9.1,2.7)--(10.8,2.7);
%\draw[fill=black](9.1,2.6)circle (0.065cm)node [above]{$\delta=-2$};
%\draw[fill=black](10.8,2.6)circle (0.065cm);
%\draw[dashed] (0.6,5.6)--(7.8,5.6);
%\draw[dashed]  (0.6,5.6)--(0.6,0.1);
%\draw[dashed] (7.8,5.6)--(7.8,0.1);
%\draw[dashed] (0.6,0.1)--(7.8,0.1);

\end{tikzpicture}
\end{center} 

With the parameter $m=\frac{X}{t}$ we obtain another fibration with singular fibers $II^*(\infty)$ (in blue), $I^*_2(0)$ (in green), $I_3(\frac{1}{4}s^2(s^2-1)^2)$ (part of it in red), $I_2(4s^4)$, $I_1(\sigma_0)$ (in yellow), where  $\sigma_{0}=-\frac{\left(  s^{2}-6s+1\right)  \left(  s^{2}+6s+1\right)
\left(  s^{2}+1\right)  ^{4}}{1728s^{2}}.$

This new fibration $\Sigma_k$ has no torsion, rank $0$, Weierstrass equation

\begin{align*}
y^{2} &  =x^{3}+2m((-s^{4}+10s^{2}-1)m+2s^{4}(s^{2}+1)^{2})x^{2}\\
&  +\left(  m-4s^{4}\right)  m^{3}\left(  (s^{8}-20s^{6}-90s^{4}%
-20s^{2}+1)\right)  x+256m^{5}s^{2}(m-4s^{4})^{2}%
\end{align*}

%Taking ... as new elliptic parameter we get another elliptic fibration $\Sigma_k$ of the $K3$ surface $S_k$, without torsion, rank $0$ and singular fibers $II^*$, $I_2^*$, $I_3$, $I_2$. Hence the N\'eron-Severi group of of this fibration is
and N\'eron-Severi group
\[NS(\Sigma_k)=U\oplus E_8 \oplus D_6 \oplus A_2 \oplus A_1.\]
By Morrison (\cite{M}, Corollary 2.10 ii), the N\'eron-Severi group of an algebraic $K3$ surface $X$ with $12\leq \rho(X) \leq 20$ is uniquely determined by its signature and discriminant form. Thus we compute $q_{NS(S_k)}$ with the help of the fibration $\Sigma_k$.
From
\[D_6^*/D_6=\langle [1]_{D_6},[3]_{D_6} \rangle \,\,\,\,\text{and}\,\,\,\, q_{D_6}([1]_{D_6})=q_{D_6}([3]_{D_6} )=-\frac{3}{2},\]

we deduce the discriminant form, since $b_{D_6}( [1]_{D_6},[3]_{D_6})=0$,
\begin{align*}
(G_{NS(S_k)},q_{NS(S_k)}) &=\mathbb Z/2 \mathbb Z(-\frac{3}{2}) \oplus \mathbb Z/2 \mathbb Z(-\frac{3}{2})\oplus \mathbb Z/3 \mathbb Z(-\frac{2}{3})\oplus \mathbb Z/2 \mathbb Z(-\frac{1}{2}) \,\,\,\text{mod.} 2\mathbb Z\\
     &=\mathbb Z/2 \mathbb Z(\frac{1}{2})\oplus \mathbb Z/6 \mathbb Z(-\frac{1}{6})\oplus \mathbb Z/2 \mathbb Z(-\frac{1}{2}).
\end{align*}
From Morrison (\cite{M} Theorem 2.8 and Corollary 2.10) there is a unique primitive embedding of $NS(S_k)$ into the $K3$-lattice $\Lambda =E_8(-1)^2 \oplus U^3$, whose orthogonal is by definition the transcendental lattice $T_{S_k}$. Now from Nikulin(\cite{Nik} Proposition 1.6.1), it follows 
\[G_{NS(S_k)}\simeq (G_{NS(S_k)})^{\perp}=G_{T_{S_k}},\,\,\,\,\,q_{T_{S_k}}=-q_{NS(s_k)}.\]
In other words the discriminant form of the transcendental lattice is
\[(G_{T_{S_k}},q_{T_{S_k}})=\mathbb Z/2 \mathbb Z(-\frac{1}{2})\oplus \mathbb Z/6 \mathbb Z(\frac{1}{6})\oplus \mathbb Z/2 \mathbb Z(\frac{1}{2}).\]

From this last relation we prove that $T_{S_k}=\langle -2 \rangle \oplus \langle 6 \rangle \oplus \langle 2 \rangle$. Denoting $T'$ the lattice $T'= \langle -2 \rangle \oplus  \langle 6 \rangle \oplus \langle 2 \rangle$, we observe that $T'$ and $T_{S_k}$ have the same signature and discriminant form. Since $|\det (T')|=24$ is small, there is only one equivalence class of forms in a genus, meaning that such a transcendental lattice is, up to isomorphism, uniquely determined by its signature and discriminant form (\cite{Co} p. 395). 

Now computing a primitive embedding of $T_{S_k}$ into $\Lambda$, since by Morrison (\cite{M} Corollary 2.10 i) this embedding is unique, its orthogonal provides $NS(S_k)$.
Take the primitive embedding $\langle(-2)\rangle=\langle e_2 \rangle \hookrightarrow E_8$, $\langle 2 \rangle=\langle u_1+u_2 \rangle \hookrightarrow U$, $ \langle 6 \rangle= \langle u_1+3u_2 \rangle \hookrightarrow U$, $(u_1,u_2)$ denoting a basis of $U$. Hence we deduce
\[NS(S_k)=U\oplus E_8(-1) \oplus E_7(-1) \oplus (-2) \oplus (-6).\]

Using their Weierstrass equations and a 2-neighbor method \cite{El}, it was proved in the previous subsection that all the fibrations $\#7$-i, $\#9$-i, $\#14$-i, $\#15$-i, $\#20$-i are on the same $K3$-surface. We can recover this result, since we know the transcendental lattice, using the Kneser-Nishiyama method. 

In that purpose, embed $T_{S_k}(-1)$ into $U\oplus E_8$ in the following way: $(-2)\oplus (-6)$ primitively embedded in $E_8$ as in Nishiyama (\cite{Nis} p. 334) and $ \langle 2 \rangle= \langle u_1+u_2 \rangle \hookrightarrow U$. We obtain $M=(T_{S_k}[-1])_{U\oplus E_8}^{\perp}=A_1\oplus A_1 \oplus A_5$. Now all the elliptic fibrations of $S_k$ are obtained from the primitive embeddings of $M$ into the various Niemeier lattices, as explained in section 2.

We identify some of these elliptic fibrations with fibrations $\#7$-i, $\#9$-i, $\#14$-i, $\#15$-i, $\#20$-i in exhibiting their torsion and infinite sections as explained in Bertin-Lecacheux \cite{BL1}, computing contributions and heights using \cite{Sc-Shio} p. 51-52. Using the Weierstrass equations given in Table 6, we compute the different local contributions and heights. Finally the identification is performed using \cite{Sc-Shio} (11.9).

\subsubsection{} Take the primitive embedding into $Ni(D_{10} E_7^2)$, given by
$A_5= \langle e_2,e_4,e_5,e_6,e_7 \rangle \hookrightarrow E_7$ and $A_1^2= \langle d_{10},d_7 \rangle \hookrightarrow D_{10}$.

Since $(A_5)_{E_7}^{\perp}=A_2$ and $(A_1^2)_{D_{10}}^{\perp}=A_1 \oplus A_1 \oplus D_6$, it follows $N=N_{\text{root}}=2A_1A_2D_6E_7$, $\det N=24\times 4$, thus the rank is $0$ and the torsion group $\mathbb Z/2 \mathbb Z$. Hence this fibration can be identified with the elliptic fibration $\#7$-i.

\subsubsection{} The primitive embedding is into $Ni(D_{10} E_7^2)$, given by 
\[A_5 \oplus A_1^2=\]
\[ \langle d_{10},d_8,d_7,d_6,d_5,d_{10}+d_9+2(d_8+d_7+d_6+d_5+d_4)+d_3,d_3 \rangle \hookrightarrow D_{10}.\]
We get 
\[(A_5 \oplus A_1^2)_{D_{10}}^{\perp}=(-6)\oplus  \langle x \rangle \oplus  \langle d_1 \rangle =(-6) \oplus A_1 \oplus A_1\]
with 
\[x=d_9+d_{10}+2(d_8+d_7+d_6+d_5+d_4+d_3+d_2)+d_1\]
and
\[(-6)=3d_9+2d_{10}+4d_8+3d_7+2d_6+d_5.\]
%and relation
%\[2[3]_{D_{10}}=(-6)+x\]

Thus $N_{\text{root}}=A_1 A_1 E_7^2$ and the rank of the fibration is $1$. Since $2[2]_{D_{10}}=x+d_1$ and there is no other relation with $[1]_{D_{10}}$ or  $[3]_{D_{10}}$, among the glue vectors $ \langle [1,1,0] \rangle, \langle [3,0,1] \rangle$ generating $Ni(D_{10} E_7^2)$, only $ \langle [2,1,1] \rangle $ contributes to torsion. 

Hence the torsion group is $\mathbb Z/2 \mathbb Z$. Moreover the 2-torsion section is
\[2F+0+[[2],[1],[1]]\]
with height $4-(1/2+1/2+3/2+3/2)=0$.
The infinite section is 
\[3F+0+[(-6),0,0]\]
with height $6$. Hence this fibration can be identified with the fibration $\#9$-i.

\subsubsection{}  The primitive embedding is into $Ni(D_{8}^3)$, 
given by $A_5= \langle d_8,d_6,d_5,d_4,d_3 \rangle \hookrightarrow D_8^{(1)}$ and $A_1
^2= \langle d_8,d_1 \rangle \hookrightarrow D_8^{(2)}$. We compute 
$(A_5)_{D_8}^{\perp}=(-6)\oplus  \langle x_1=(-2) \rangle \oplus \langle d_1 \rangle$ with 
$x_1=d_7+d_8+2(d_6+d_5+d_4+d_3+d_2)+d_1$
\[
\begin{matrix}
(A_1^2)_{D_8}^{\perp}= & \langle d_7 \rangle \oplus \langle x_1=d_7+d_8+2(d_6+d_5+d_4+d_3+d_2)+d_1 \rangle \\
  & \oplus \langle d_5,d_4,x_3=d_7+d_8+2d_6+d_5,d_3 \rangle =A_1\oplus A_1 \oplus D_4.
\end{matrix}
\]
We deduce $N_{\text{root}}=4A_1D_4D_8$ (hence the fibration has rank $1$) and the relations
\begin{align}
2[2]_{D_8}=x_1+d_1\\
2([2]_{D_8}-(d_1+d_2))=x_3+2d_3+2d_4+d_5\\
2[3]_{D_8}=x_1+2x_3+d_3+2d_4+d_5+d_7\\
2([1]_{D_8}-(d_6+d_8))=x_1+x_3+d_3+2d_4+2d_5+d_7\\
2([1]_{D_8}-(d_6+d_7+d_8))=2x_3+3d_5+4d_4+3d_3+2d_2+d_1-d_7.
\end{align}
Thus, among the glue vectors $ \langle [1,2,2],[1,1,1],[2,2,1] \rangle $ generating the Niemeier lattice, only vectors $ \langle [0,3,3],[2,1,2] \rangle $ contribute to torsion and the torsion group is $\mathbb Z/2 \mathbb Z \times \mathbb Z /2 \mathbb Z$.

From relations $(1)$ to $(5)$ we deduce the various contributions and heights of the following sections (see Table 5.1.3).
\begin{table}

\begin{center}
\begin{tabular}{|cccccccc|}

\hline
     &  Cont.      &  Cont.   &  Cont.  & Cont. &  Cont. & Cont. &ht. \\
$0+2F+$   &   $A_1$   &  $ A_1$ &  $A_1$ & $A_1$ & $D_4$  & $D_6$  &  \\
\hline
 ${\scriptstyle [0, [3],[3]]}$ & $0$    & 0  & 1/2  & 1/2 & 1  & 2 & 0 \\
%\hline

${\scriptstyle [[2],[2]-(d_1+d_2),[1]]}$   & 1/2 & 1/2 & 0 & 0 & 1 & 2 & 0 \\
%\hline
${\scriptstyle [[2],[1]-(d_6+d_8),[2]]}$     & 1/2 & 1/2 & 1/2 & 1/2 & 1 & 1 & 0 \\
%\hline
${\scriptstyle  [[2]-(d_1+d_2),[1]-(d_6+d_7+d_8),[2]]}$    & 0 & 0 & 1/2 & 0 & 1 & 1 & 3/2 \\
\hline
\end{tabular}
\end{center}
\caption{Contributions and heights of the sections of 5.1.3}
\end{table}

Hence this fibration can be identified with the fibration $\#14$-i.

\subsubsection{} The primitive embedding is into $Ni(D_8^3)$ and given by 
\[A_5= \langle d_8,d_6,d_5,d_4,d_3 \rangle \hookrightarrow D_8^{(1)}\,\,\,\,A_1= \langle d_8 \rangle \hookrightarrow D_8^{(2)}\,\,\,\,A_1= \langle d_8 \rangle \hookrightarrow D_8^{(3)}.\]

As previously $(A_5)_{D_8}^{\perp}= \langle (-6) \rangle \oplus  \langle x_1 \rangle \oplus \langle d_1 \rangle $; we get also $ \langle d_8 \rangle_{D_8}^{\perp}= \langle d_7 \rangle \oplus  \langle x_4=d_7+d_8+2d_6+d_5,d_5,d_4,d_3,d_2,d_1 \rangle =A_1 \oplus D_6$.
Hence $N_{\text{root}}=4A_12D_6$, and the rank is $1$. Moreover it follows the relations
\begin{align}
2[2]_{D_8}=x_1+d_1\\
2[2]_{D_8}=x_3+d_5+2d_4+2d_3+2d_2+2d_1\\
2([1]_{D_8}-(d_5+d_6+d_7+d_8))=2x_3+d_5+4d_4+3d_3+2d_2+d_1-d_7\\
2[3]_{D_8}=3x_3+d_7+2d_5+4d_4+3d_3+2d_2+d_1 \in A_1 \oplus D_6.
\end{align}
We deduce that among the glue vectors generating $Ni(D_8^3)$, only $ \langle [0,3,3],[2,1,2] \rangle $ contribute to torsion. So the torsion group is $\mathbb Z/2 \mathbb Z \times \mathbb Z /2 \mathbb Z$.
From relations $(6)$ to $(9)$ we deduce the various contributions and heights of the following sections (see Table 5.1.4).

\begin{table}

\begin{center}
\begin{tabular}{|cccccccc|}

\hline
     &  Cont.      &  Cont.   &  Cont.  & Cont. &  Cont. & Cont. &ht. \\
$0+2F+$   &   $A_1$   &  $ A_1$ &  $A_1$ & $D_6$ & $A_1$  & $D_6$  &  \\
\hline
 ${\scriptstyle [0, [3],[3]]}$ & $0$    & 0  & 1/2  & 1+1/2 & 1/2  & 1+1/2 & 0 \\
%\hline

${\scriptstyle [[2],[1]-(d_5+d_6+d_7+d_8),[2]]}$   & 1/2 & 1/2 & 1/2 & 1+1/2 & 0 & 1 & 0 \\
%\hline
${\scriptstyle [[2],[2],[1]-(d_5+d_6+d_7+d_8)]}$     & 1/2 & 1/2 & 0 & 1 & 1/2 & 1+1/2 & 0 \\
%\hline
${\scriptstyle  [[3], 0,[3]]}$    & 0 & 1/2 & 0 & 0 & 1/2 & 1+1/2 & 3/2 \\

\hline

\end{tabular}
\end{center}
\caption{Contributions and heights of the sections of 5.1.4}
\end{table}

Hence this fibration can be identified with the fibration $\#15$-i.

\subsubsection{} The primitive embedding is onto $Ni(A_5^4D_4)$ given by $A_5 \hookrightarrow A_5$, $A_1 \oplus A_1 = \langle d_4,d_1> \rangle \hookrightarrow D_4$.
Since $ \langle d_4,d_1 \rangle_{D_4}^{\perp}=A_1^2$, we get $N=N_{\text{root}}=3A_5 2A_1$; thus the rank of the fibration is $0$ and since $\det(N)=24 \times 6^2$, the torsion group is $\mathbb Z/6 \mathbb Z$.

This fibration can be identified with the fibration $\#20$-i.

\end{proof}

\begin{remark}
From fibration $\#20-i$ the surface $S_k$ appears to be a double cover of the rational elliptic modular surface associated to the modular groupe $\Gamma_0(6)$ given in Beauville's paper \cite{Beau}
\[
(x+y)(y+z)(z+x)(t-s)(ts-1)=8sxyz.
\] 
\end{remark}

\section{Proof of Theorem $1.2$}

We recall first on Table \ref{Ta:Fib} the results obtained by Bertin and Lecacheux in \cite{BL}. The notation $\# 17 (18-m)$ for example refers  for $\#17$ to the generic case when relevant and for $(18-m)$ to notations used in Bertin-Lecacheux \cite{BL}.

 \begin{table}[tp]\footnotesize
%\caption{The elliptic fibrations of $Y_2$}
\begin{center}
\begin{tabular}{|l|l|l|l|l|l|l|}
\hline
$L_{\text{root}}$ & $L/L_{\text{root}}$ &  & & \text{Fibers} & \text{R} & \text{Tor.} \\ \hline
  $E_8^3$    &  $(0)$ & & & & &\\ \hline
& $\#1(11-f)$  & $A_1\subset E_8$ & $D_5 \subset E_8$  & $E_7 A_3 E_8$ & $0$ & $(0)$\\ \hline
& $\#2(13-h)$ & $A_1\oplus D_5 \subset E_8$ &   & $A_1 E_8  E_8$ & $1$ & $(0)$\\ \hline
  $E_8 D_{16}$    &  $\mathbb Z /{2 \mathbb Z}$ & & & & &\\ \hline
& $ \#3(30-\phi)$  & $A_1\subset E_8$ & $D_5 \subset D_{16}$  & $E_7 D_{11}$ & $0$ & $(0)$\\ \hline
& $\#4(16-o)$ & $A_1\oplus D_5 \subset E_8$ &   & $A_1 D_{16}$ & $1$ & $\mathbb Z /{2 \mathbb Z}$\\ \hline
& $\#5(17-q)$  & $D_5\subset E_8$ & $A_1 \subset D_{16}$  & $A_3 A_1 D_{14}$ & $0$ & $\mathbb Z /{2 \mathbb Z}$\\ \hline
&  $\#6(25-\delta)$ & $A_1\oplus D_5 \subset D_{16}$ &   & $E_8 A_1 D_{9}$ & $0$ & $(0)$\\ \hline
  $E_7^2 D_{10}$    &  $(\mathbb Z /{2 \mathbb Z})^2$ & & & & &\\ \hline
& $\#7(29-\beta)$  & $A_1\subset E_7$ & $D_5 \subset D_{10}$  & $E_7 D_6 D_5$ & $0$ & $\mathbb Z /{2 \mathbb Z}$\\ \hline
& $\#8(9-r)$  & $A_1\subset E_7$ & $D_5 \subset E_7$  & $ D_6 A_1 D_{10}$ &$1$ & $\mathbb Z /{2 \mathbb Z}$\\ \hline
& $\#8\text{bis}(24-\psi)$  & $A_1\oplus D_5 \subset E_7$ &   & $E_7 D_{10}$ & $1$ & $\mathbb Z /{2 \mathbb Z}$\\ \hline
&  $\#9(12-g)$ & $A_1\oplus D_5 \subset D_{10}$ &   & $E_7 E_7 A_1 A_3$ & $0$ & $\mathbb Z /{2 \mathbb Z}$\\ \hline
& $\#10(10-e)$ & $D_5\subset E_7$ & $A_1 \subset D_{10}$  & $A_1 A_1 D_8 E_7$ & $1$ & $\mathbb Z /{2 \mathbb Z}$\\ \hline
 $E_7 A_{17}$    &  $\mathbb Z /{6 \mathbb Z}$ & & & & &\\ \hline
&  $(21-c)$ & $A_1\oplus D_5 \subset E_7$ &   & $A_{17}$ & $1$ & $\mathbb Z /{3 \mathbb Z}$\\ \hline
& $\#11(19-n)$  & $D_5\subset E_7$ & $A_1 \subset A_{17}$  & $A_1 A_{15}$ & $2$ & $(0)$\\ \hline
 $D_{24}$    &  $\mathbb Z /{2 \mathbb Z}$ & & & & &\\ \hline
& $\#12(23-i)$ & $A_1\oplus D_5 \subset D_{24}$ &   & $A_1 D_{17}$ & $0$ & $(0)$\\ \hline

 $D_{12}^2$    &  $(\mathbb Z /{2 \mathbb Z})^2$ & & & & &\\ \hline
& $\#13(26-\pi)$  & $A_1 \subset D_{12}$ & $D_5 \subset D_{12}$  & $A_1 D_{10} D_{7}$ & $0$ & $\mathbb Z /{2 \mathbb Z}$\\ \hline
& $\#14(22-u)$  & $A_1\oplus D_5 \subset D_{12}$ &   & $A_1 D_5 D_{12}$ & $0$ & $\mathbb Z /{2 \mathbb Z}$\\ \hline
 $D_8^3$    &  $(\mathbb Z /{2 \mathbb Z})^3$ & & & & &\\ \hline
&  $\#15(6-p)$ & $A_1 \subset D_{8}$ & $D_5 \subset D_{8}$  & $A_1 D_{6}A_3 D_{8}$ & $0$ & $(\mathbb Z /{2})^2 $\\ \hline
&  $\#16(14-t)$ & $A_1\oplus D_5 \subset D_{8}$ &   & $A_1 D_8 D_{8}$ & $1$ & $\mathbb Z /{2 \mathbb Z}$\\ \hline
 $D_9 A_{15}$    &  $\mathbb Z /{8 \mathbb Z}$ & & & & &\\ \hline
& $\#17(18-m)$  & $A_1\oplus D_5 \subset D_{9}$ &   & $ A_1 A_1 A_1 A_{15}$ & $0$ & $\mathbb Z /{4 \mathbb Z}$\\ \hline
&  $\#18(28-\alpha)$ & $D_5 \subset D_{9}$ & $A_1 \subset A_{15}$  & $D_4 A_{13}$ & $1$ & $(0)$\\ \hline
 $E_6^4$    &  $(\mathbb Z /{3 \mathbb Z)^2}$ & & & & &\\ \hline
& $\#19(8-b)$  & $A_1 \subset E_6$ & $D_5 \subset E_6$  & $A_5 E_6 E_6$ & $1$ & $\mathbb Z /{3 \mathbb Z}$\\ \hline
 $ A_{11} E_6 D_7$    &  $\mathbb Z /{12 \mathbb Z}$ & & & & &\\ \hline
&  $\#20(7-w)$ & $A_1 \subset E_6$ & $D_5 \subset D_7$  & $A_5 A_1 A_1 A_{11}$ & $0$ & $\mathbb Z /{6 \mathbb Z}$\\ \hline
&  $\#21(27-\mu)$ & $A_1 \subset A_{11}$ & $D_5 \subset D_7$  & $A_9 A_1 A_1 E_6$ & $1$ & $(0)$\\ \hline
&  $(20-j)$ & $A_1\oplus D_5 \subset D_{7}$ &   & $A_{11}E_6  A_1 $ & $0$ & $\mathbb Z /{3 \mathbb Z}$\\ \hline
& $\#22(15-l)$  & $A_1 \subset A_{11}$ & $D_5 \subset E_6$  & $ A_9 D_7$ & $2$ & $(0)$\\ \hline
& $\#23(2-k)$  & $D_5 \subset E_6$ & $A_1 \subset D_7$  & $ A_{11} A_1 D_5$ & $1$ & $\mathbb Z /{4 \mathbb Z}$\\ \hline
 $D_6^4$    &  $(\mathbb Z /{2 \mathbb Z})^4$ & & & & &\\ \hline
&  $\#24(5-d)$ & $A_1 \subset D_6$ & $D_5 \subset D_6$  & $ A_1 D_4 D_6 D_6$ & $1$ & $(\mathbb Z /2)^2$\\ \hline
 $D_6 A_9^2$    &  $\mathbb Z /{2}\times \mathbb Z /{10} $ & & & & &\\ \hline
&  $\#25(3-v)$ & $D_5 \subset D_6$ & $A_1 \subset A_9$  & $ A_{7} A_9$ & $2$ & $(0)$\\ \hline

 $D_5^2 A_7^2$    &  $\mathbb Z /{4}\times \mathbb Z /{8} $ & & & & &\\ \hline
&  $\#26(1-s)$ & $D_5 \subset D_5$ & $A_1 \subset D_5$  & $A_1  A_{3} A_7 A_7$ & $0$ & $\mathbb Z /{8 \mathbb Z}$\\ \hline
& $\#27(4-a)$ & $D_5 \subset D_5$ & $A_1 \subset A_7$  & $D_5 A_5 A_7 $ & $1$ & $(0)$\\ \hline

\end{tabular}
\end{center}
\caption{The elliptic fibrations of $Y_2$}\label{Ta:Fib}
\end{table}

Comparing to the fibrations of the family you remark more elliptic fibrations with $2$-torsion sections on $Y_2$. All the corresponding involutions are denoted $\tau$. Some of them are specialisations for $s=1$ of the generic ones. Those generic which are Morrison-Nikulin still remain Morrison-Nikulin for $Y_2$ by a Sch\"{u}tt's lemma \cite{Sc}, namely $\#4-\tau$, $\#8-\tau$, $\#16-\tau$, $\#17-\tau$, $\#23-\tau$, $\#24 a)-\tau$, $\#26-\tau$. Others (($\#5-\tau$, $\#8\text{bis}-\tau$, $\#10-\tau$, $\#15 <(p,0>)-\tau$, $\#24\, b-\tau$, $\#24\,c-\tau$) are specific to $K_2$ and cannot be deduced from elliptic fibrations of the generic Kummer. To identify them we have to use the distinguished property of $Y_2$, that is $Y_2$ is a singular $K3$ with Picard number $20$. 

Hence $Y_2$ inherits of a Shioda-Inose structure, that is the quotient of $Y_2$ by an involution is isomorphic to a Kummer surface $K_2$ realized from the product of CM elliptic curves \cite{SI}, \cite{SM} provided in the following way.

Since the transcendental lattice of $Y_2$ is $\mathbb T (Y_2)= \begin{pmatrix}                                                                                         2 & 0 \\
                                                   0 & 4
                                                 \end{pmatrix}=\begin{pmatrix}

                                                   2a & b \\
                                                   b & 2c
                                                 \end{pmatrix}$

we get $b^2-4ac=-8$, $\tau_1=\frac{-b+\sqrt{b^2-4ac}}{2a}$, $\tau_2=\frac{b+\sqrt{b^2-4ac}}{2}$, hence $\tau=\tau_1=\tau_2=i \sqrt{2}$.

We deduce $K_2=E\times E/\pm 1$ with $E=\mathbb C/(\mathbb Z+\tau \mathbb Z)$ and $j(E)=j(i\sqrt{2})=8000$. The fact that the two CM elliptic curves are equal and satisfy $j(E)=8000$ can be obtained also by specialisation from the Shioda-Inose structure of the family. (see \ref{3} Remark 4.1 ).

The elliptic curve $E$ can be also put in the Legendre form:

\[E \,\,\,\,\,y^2=x(x-1)(x-l),\]
$l$ satisfying the equation $j=8000=\frac{256(1-l+l^2)^3}{l^2(l-1)^2}$. Thus $l=3\pm 2\sqrt{2}$ or $l=-2\pm 2\sqrt{2}$ or $l=\frac{1\pm \sqrt{2}}{2}$.

\begin{proposition}\label{prop}
The elliptic fibrations $\#8 \text{bis}-\tau$ and $\#10-\tau$ are elliptic fibrations of $K_2$.
\end{proposition}
\begin{proof}

It follows from the $4.2$ fibration $F_8$ that the fibration $\#10-\tau$ with Weierstrass equation
\[Y^2=X^3-2U^2(U-1)X^2+U^3(U+1)^2(U-4)X,\]
singular fibers $III^*(0)$, $I_2^*(\infty)$, $I_4(-1)$, $I_2(4)$, $I_1(-1/2)$, and $\mathbb Z /2 \mathbb Z$-torsion is an elliptic fibration of $K_2$. Similarly from  the $4.2$ fibration $G_8$, we deduce that the elliptic fibration  $\#10-\tau$ with Weierstrass equation
\[Y^2=X^3+2(t+5t^2)X^2+t^2(4t+1)(t^2+6t+1)X,\]
singular fibers $III^*(\infty)$, $I_3^*(0)$, $3I_2(-1/4,t^2+6t+1=0)$, and $\mathbb Z /2 \mathbb Z$-torsion is an elliptic fibration of $K_2$.
\end{proof}

%Kuwata and Shioda \cite{KS} gave different type of elliptic fibrations on a Kummer surface associated to the product of elliptic curves $E_1 \times E_2$ in the Legendre form,  more precisely the type $\mathcal I_8$, with Weierstrass equation
%\[Y^2=X^3-u((2l_1l_2-l_1-l_2+2)u-2)X^2-u^2(u-1)(l_1l_2u-1)((l_1-1)(l_2-1)u-1)X\]
 
%\begin{enumerate}
%\item
%The choice $l_1=3+2\sqrt{2}$ and $l_2=3-2\sqrt{2}$ gives the Weierstrass equation 
%\[Y^2=X^3+2u(u+1)X^2+u^2(u-1)^2(4u+1)X\]
%or with the change of variable $u=1/U$
%\[Y^2=X^3-2U^2(U-1)X^2+U^3(U+1)^2(U-4)X\]
%%with singular fibers $III^*$ $I_2^*$ $I_4$ $I_2$ and torsion group $\mathbb Z /2 \mathbb Z $. 
%\item
%The choice $l_1=l_2=3+2\sqrt{2}$ gives the second elliptic fibration.

%\end{enumerate}

To achieve the proof of Theorem 1.2 we need also the following lemma.
\begin{lemma}\label{lem:SZ}
The Kummer $K_2$ has exactly $4$ extremal elliptic fibrations given by Shimada Zhang \cite{SZ} with the type of their singular fibers and their torsion group 
\begin{enumerate}

\item $E_7$ $A_7$ $A_3$ $A_1$  $\mathbb Z /2 \mathbb Z$,

\item $D_9$ $A_7$ $A_1$ $A_1$, $\mathbb Z /2 \mathbb Z$,

\item $D_6$ $D_5$ $A_7$, $\mathbb Z /2 \mathbb Z$,

\item $A_7$ $A_3$ $A_3$ $A_3$ $A_1$ $A_1$, $\mathbb Z /2 \mathbb Z \times \mathbb Z /4 \mathbb Z $.
\end{enumerate}

\end{lemma}

From lemma (\ref{lem:SZ} (2)) we obtain the fibration $\#5-\tau$ and from lemma (\ref{lem:SZ} (3)) the fibration $\#15 \langle (p,0)\rangle -\tau$.

We notice also that fibrations $\#17-\tau$ and $\#26-\tau$ obtained by specialisation are also fibration $(4)$ of lemma (\ref{lem:SZ}) and fibration $\#23- \tau$, by a $2$-neighbor process of parameter $m=\frac{X}{k^2(k^2+4)}$ gives fibration $(3)$ of lemma (\ref{lem:SZ}).
  
Finally, by a $2$-neighbor process of parameter $m=\frac{X}{d^2(d+1)}$, fibrations $\#24 b)-\tau$ and $\#24 c)-\tau$ gives fibration $\#16-\tau$, hence are elliptic fibrations of $K_2$.

\begin{corollary}
As a byproduct of the proof we get Weierstrass equations for extremal fibrations of lemma (\ref{lem:SZ}) (2), (3), (4).

\end{corollary}

The table of symplectic automorphisms of order $2$ (self involutions) results from an easy computation.

\begin{table}[tp]\footnotesize
\[%
\begin{array}
[l]{|l|l|l|}\hline
\text{No} & \text{Weierstrass Equation} & \text{From or to}\\\hline

\#4 &   

\begin{array}
[l]{l}%
%\begin{array}
%[c]{c}%
   y^2=x^3+(o^3-5o^2+2)x^2+x \\
   \hline
I_{12}^*(\infty) , I_2 (0), 4 I_1(1,5,o^2-4o-4)\\
\hline
%\end{array}\\ %\hline
   Y^2=X(X-o^3+5o^2)(X-o^3+5o^2-4)\\
   \hline
 I_6^*(\infty), I_4 (0), 4I_2(1,5,o^2-4o-4)\\
%\text{lemma}2c)
% \\%\hline
\end{array}
   & \text{Spec.} \#4-i \\\hline \hline
\#5 &
\begin{array}
[l]{l}%
%\begin{array}
%[c]{c}%
 y^2=x^3+(q^3+q^2+2q-2)x^2+(1-2q)x\\
 \hline
I_{10}^* (\infty), I_4 (0),I_2 (\frac{1}{2}),2I_1 (q^2+2q+5) \\
%\hline
%\end{array}
%\\
 \hline
Y^2=X^3-2(q^3+q^2+2q-2)X^2+Xq^4(q^2+2q+5)\\
\hline
I_5^* (\infty), I_8 (0), 2I_2 (q^2+2q+5),I_1 (\frac{1}{2})\\
%\text{lemma} 1)b) \\
%\hline

\end{array}
  & \text{lemma}(\ref{lem:SZ})(2)\\\hline \hline

\#8 &
\begin{array}
[l]{l}%
%\begin{array}
%[c]{c}%
  y^2=x^3-r(r^2-r+2)x^2+r^3x\\
  \hline
I_6^*(\infty), I_2^*(0), I_2(1),2I_1(\pm 2i)\\
\hline
%\hline
%\end{array}\\ \hline
  Y^2=X^3+2r(r^2-r+2)X^2+(r-1)^2r^2(r^2+4)X\\
\hline
I_3^*(\infty), I_1^*(0), I_4(1),2I_2(\pm 2i)\\
% \text{lemma} 2)h)   \\
%\hline
\end{array}
  
&  \text{Spec.} \#8-i \\\hline \hline

 \#8bis  &
\begin{array}
[l]{l}%
%\begin{array}
%[c]{c}%
 y^2=x^3-(\psi+5\psi^2)x^2-\psi^5x \\
 \hline
I_{6}^*(\infty),III^*(0),3I_1(-\frac{1}{4},\psi^2+6\psi+1)\\
\hline %\hline
%\end{array}\\ \hline
  Y^2=X^3+2(5\psi^2+\psi)X^2+X\psi^2(4\psi+1)(\psi^2+6\psi+1)\\
  \hline
 III^*(\infty),I_3^*(0),3I_2(-(\frac{1}{4}),\psi^2+6\psi+1) \\ 
%\text{lemma} 2)f)   \\
%\hline
\end{array}
& \text{lemma}(\ref{prop})

\\\hline \hline
 \#10 &
\begin{array}
[l]{l}%
%\begin{array}
%[c]{c}%
y^2=x(x^2-e^2(e-1)x+e^3(2e+1))\\\hline
I_4^*(\infty),III^*(0), 2I_2(-1,-\frac{1}{2}),I_1(4) \\
\hline %\hline
%\end{array}\\ \hline
  Y^2=X^3+2e^2(e-1)X^2+e^3(e-4)(e+1)^2X \\\hline
I_2^*(\infty),III^*(0),I_4(-1),I_2(4),I_1(-\frac{1}{2})\\
% \text{lemma} 2)b)  \\
%\hline
\end{array}
& \text{lemma}(\ref{prop})

\\\hline \hline
\begin{array}
[l]{l}
 \#15 \\
<(p,0)> \\
\end{array}
  & 
\begin{array}
[l]{l}%
%\begin{array}
%[c]{c}%
 y^2=x(x-p)(x-p(p+1)^2) \\
 \hline
 I_4^*(\infty),I_2^*(0),I_4(-1),I_2(-2)  \\
\hline 
%\end{array}\\ \hline
 %<p,0> \\
Y^2=X^3-p(p^2+2p-1)X^2-p^3(p+2)X \\
\hline
 I_2^*(\infty),I_1^*(0),I_8(-1),I_1(-2) \\
 %\hline
 % \text{lemma} 1)c)   \\
%\hline

\end{array}
& \text{lemma}(\ref{lem:SZ})(3)
\\\hline \hline

 \#16  &
\begin{array}
[l]{l}%
%\begin{array}
%[c]{c}%
 y^2=x^3+t(t^2+4t+1)x^2+t^4x \\
 \hline
2I_4^*(\infty,0),I_2(-1),2I_1(t^2+6t+1=0) \\
\hline %\hline
%\end{array}\\ \hline
  Y^2=X^3-2t(t^2+4t+1)X^2+t^2(t+1)^2(t^2+6t+1)X \\
  \hline
 2I_2^*(\infty,0),I_4(-1),2I_2(t^2+6t+1) \\
 % \text{ lemma} 2)a)  \\
%\hline
\end{array}
& \text{Spec.}\#16-i
\\\hline \hline

\#17  &
\begin{array}
[l]{l}%
%\begin{array}
%[c]{c}% 
 y^2=x(x^2+x(\frac{1}{4}(m^2-4)^2-2)+1) \\
 \hline
I_{16}(\infty),3I_2(0,\pm 2),2I_1(\pm 2\sqrt{2}) \\
\hline %\hline
%\end{array}\\ \hline
 Y^2=X(X-\frac{1}{4}m^4+2m^2)(X-\frac{1}{4}m^4+2m^2-4)\\
\hline
I_8(\infty),3I_4(0,\pm 2),2I_2(\pm 2\sqrt{2}) \\
%  \text{lemma} 1)d) \\
%\hline
\end{array}
& \begin{matrix}
\text{lemma}(\ref{lem:SZ})(4)\\
 \text{Spec.}\#17-i
\end{matrix}
\\\hline \hline

\#23  &
\begin{array}
[l]{l}%
%\begin{array}
%[c]{c}% 
 y^2=x^3+x^2(\frac{1}{4}k^4-k^3+k^2-2k)+k^2x \\
 \hline
I_{12}(\infty),I_1^*(0),I_2(2),3I_1(4,\pm 2i) \\
\hline %\hline
%\end{array}\\ \hline
%<0,0> \\
  Y^2=X^3-(\frac{1}{2}k^4-2k^3-4k)X^2+\frac{k^3(k-4)(k^2+4)(k-2)^2}{16}X \\
  \hline
I_6(\infty),I_2^*(0), I_4(2), 3 I_2(4,\pm 2i) \\
%\hline
%  \text{lemma} 2)e)  \\
%\hline
%\end{array}
\end{array}
 & \begin{matrix}
\frac{X}{k^2(k^2+4} \\
 \text{to}\,\,\, \text{lemma}(\ref{lem:SZ})(3)
\end{matrix}
\\
\hline \hline
%\end{array}
\begin{array}
[l]{l} 
\#24 \\
a)\langle(0,0)\rangle\\
b)\langle(d+d^2,0)\rangle\\
c)\langle(d^2+d^3,0)\rangle\\
\end{array}
& 
\begin{array}
[l]{l}%
%\begin{array}
%[c]{c}% 
y^2=x(x-(d+d^2))(x-(d^3+d^2))   \\ \hline
2I_2^*(\infty,0),I_0^*(-1),I_2(1) \\
\hline
%\hline
%\end{array}\\ \hline
%a) <(0,0> \\
a): Y^2=X^3+2d(d+1)^2X^2+d^2(d^2-1)^2X \\ \hline
2I_1^*(\infty,0),I_0^*(-1),I_4(1) \\
  %\text{lemma} 2)g)   

\hline 
% b) <(d+d^2,0)> \\
 b): Y^2=X^3+2d(d+1)(d-2)X^2+d^4(d+1)^2X \\ \hline
I_1^*(\infty),I_4^*(0),I_0^*(-1),I_1(1) \\  %\hline
 %\text{lemma} 2)d)  \\
%\hline
c):Y^2=X^3-2d(d+1)(2d-1)X^2+Xd^2(d+1)^2\\
\end{array}

& 
\begin{array}
[l]{l}
 a) \text{Spec.}\#24-i\\
 b) m=X/d^2(d+1)\\
\text{to}\,\,\, \#16-\tau\\
c) \text{ similar to } b)
\end{array}

\\\hline \hline

\#26  & 

\begin{array}
[l]{l}%
%\begin{array}
%[c]{c}% 
y^2=x^3+x^2(\frac{1}{4}(s-1)^4-2s^2)+s^4x  \\
\hline
2I_8(\infty,0),I_4(1),I_2(-1),2I_1(3\pm 2\sqrt{2}) \\
\hline
%\hline
%\end{array}\\ \hline
  Y^2=X(X-\frac{1}{4}(s-1)^4+4s^2)(X-\frac{1}{4}(s-1)^4) \\ \hline
I_8(1),3I_4(0,-1,\infty ),2I_2( 3\pm 2\sqrt{2}) \\
 %\text{lemma} 1)d)    \\
%\hline
\end{array}
& \begin{matrix}
\text{Spec.} \#26-i\\
\text{lemma}(\ref{lem:SZ})(4)
\end{matrix}
\\\hline %\hline 
\end{array}
\]
\caption{Morrison-Nikulin involutions of $Y_2$ (fibrations of $K_2$)}\label{Ta:IsY_2}
\end{table}

\begin{table}[tp]\footnotesize
\[%
\begin{array}
[c]{|c|c|}\hline
\text{No} & \text{Weierstrass Equation} \\\hline

\#7  &
\begin{array}
[c]{c}%
%\begin{array}
%[c]{c}%
   y^2=x^3+2\beta^2(\beta-1)x^2+\beta^3(\beta-1)^2x \\
   \hline
 I_2^*(\infty),III^*(0),I_1^*(-1)\\
\hline
%\hline
%\end{array}\\ \hline
 Y^2=X^3-4\beta^2(\beta-1)X^2+4\beta^3(\beta-1)^3X \\
 \hline
I_1^*(\infty),III^*(0),I_2^*(1)\\
%\text{self}  \\
%\hline
\end{array}
\\
\hline \hline

\#9  & 
\begin{array}
[c]{c}%
%\begin{array}
%[c]{c}%
y^2=x^3+4g^2x^2+g^3(g+1)^2x  \\
\hline
2III^*(\infty,0),I_4(-1),I_2(1)  \\
\hline
%\hline
%\end{array}\\ \hline
Y^2=X^3-8g^2X^2-4g^3(g-1)^2X  \\
\hline
2III^*(\infty,0),I_4(1),I_2(-1) \\
 %\text{self}  \\
%\hline
\end{array}
\\\hline \hline 

\#13 &
\begin{array}
[c]{c}%
%\begin{array}
%[c]{c}%
  y^2=x^3+x^2\pi(\pi^2-2\pi-2)+\pi^2(2\pi+1)x  \\
  \hline
I_6^*(\infty),I_3^*(0), I_2(-1/2),I_1(4) \\
\hline
%\hline
%\end{array}\\ \hline
 Y^2=X^3-2X^2\pi(\pi^2-2\pi-2)+\pi^5(\pi-4)X  \\
 \hline
I_6^*(0),I_3^*(\infty), I_2(4),I_1(-1/2)  \\
% \text{self}    \\
%\hline
\end{array}
\\ \hline \hline

 \#14 & 
\begin{array}
[c]{c}%
%\begin{array}
%[c]{c}%
 y^2=x^3+u(u^2+4u+2)x^2+u^2x  \\
 \hline
I_8^*(\infty),I_1^*(0),I_2(-2),I_1(-4) \\
\hline
%\hline
%\end{array}\\ \hline
 % Y^2=X^3-2u(u^2+4u+2)X^2+u^3(u^3+8u^2+20u+16)X  \\
Y^2=\left (X-u\left (u-2\right )^{2}\right )X\left (X-4\,u\right )\\ 
 \hline
I_4^*(\infty),I_2^*(0),I_4(-2),I_2(-4) \\
%   \#15   \\
\end{array}
\\ \hline \hline 
\begin{array}
[c]{c}
\#15 \\
a)<(0,0)>\\
b)<(p(p+1)^2,0)>\\
   
\end{array}& 
\begin{array}
[c]{c}%
%\begin{array}
%[c]{c}%
 y^2=x(x-p)(x-p(p+1)^2) \\
 \hline
I_4^*(\infty),I_2^*(0),I_4(-1),I_2(-2)  \\
\hline
%\hline
%\end{array}\\ \hline
%<0,0> \\
a):Y^2=X(X+4p+p^3+4p^2)(X+p^3)  \\
\hline
I_4^*(0),I_2^*(\infty),I_4(-2),I_2(-1)\\
%  \text{self} \\
%\end{array}
%\\
 \hline \hline

% \#15*  & 
%\begin{array}
%[c]{c}%
%\begin{array}
%[c]{c}%
%y^2=x(x-p)(x-p(p+1)^2)  \\
%\hline
%I_4^*(\infty),I_2^*(0),I_4(-1),I_2(-2)  \\
%\hline
%\hline
%\end{array}\\ \hline
%<p(p+1)^2,0> \\
b):  Y^2=X^3-2p(2p^2+4p+1)X^2+p^2X \\
  \hline
I_8^*(\infty),I_1^*(0),I_2(-1),I_1(-2)\\
%    \#14    \\
\end{array}
\\ \hline \hline 

\#20  & 
\begin{array}
[c]{c}%
%\begin{array}
%[c]{c}%
y^2=x^3-(2-w^2-\frac{1}{4}w^4)x^2-(w^2-1)x  \\
\hline
I_{12}(\infty),I_6(0),2I_2(\pm 1),2I_1(\pm2i\sqrt{2}) \\
\hline
%\hline
%\end{array}\\ \hline
Y^2=X^3+2(2-w^2-\frac{1}{4}w^4)X^2+\frac{1}{16}w^6(w^2+8)X  \\
\hline
I_{12}(0),I_6(\infty),2I_2(\pm 2i\sqrt{2}),2I_1(\pm 1) \\
%  \text{self} \\
\end{array}
\\ \hline % \hline

\end{array}
\]

\caption{Self involutions of $Y_2$ }\label{Ta:IsY_4}
\end{table}

\section{$2$-isogenies and isometries}

Theorem 1.2, where the 2-isogenous $K3$ surfaces of $Y_2$ are either its Kummer $K_2$ or $Y_2$ itself, cannot be generalised to all the other singular $K3$ surfaces of the Ap\'ery-Fermi family. The reason is the relation with a Theorem of Boissi\`ere, Sarti and Veniani \cite{BSV}, telling when $p$-isogenies ($p$ prime) between complex projective $K3$ surfaces $X$ and $Y$ define isometries between their rational transcendental lattices $T_{X,\mathbb {Q}}$ and $T_{Y,\mathbb {Q}}$. ( These lattices are isometric if there exists $M\in \text{Gl}(n,\mathbb Q)$ satisfying $T_{X,\mathbb {Q}}=M^tT_{Y,\mathbb {Q}}M$. Let us recall the part of their Theorem related to $2$-isogenies.

\begin{theorem} \cite{BSV}
Let $\gamma: X \rightarrow Y$ be a $2$-isogeny between complex projective $K3$ surfaces $X$ and $Y$. Then $\text{rk}( T_{Y,\mathbb {Q}})=\text{rk}( T_{X,\mathbb {Q}})=:r$ and 
\begin{enumerate}
\item If $r$ is odd, there is no isometry between $T_{Y,\mathbb {Q}}$ and $T_{X,\mathbb {Q}}$.
\item If $r$ is even, there exists an isometry between $T_{Y,\mathbb {Q}}$ and $T_{X,\mathbb {Q}}$ if and only if $T_{Y,\mathbb {Q}}$ is isometric to $T_{Y,\mathbb {Q}}(2)$. This property is equivalent to the following: for every prime number $q$ congruent to $3$ or $5$ modulo $8$, the $q$-adic valuation $\nu_q(\det T_Y)$ is even.
\end{enumerate}

\end{theorem}

As a Corollary we deduce the following result.

\begin{theorem}
Among the singular $K3$ surfaces of the Ap\'ery-Fermi family defined for $k$ rational integer, only $Y_2$ and $Y_{10}$ possess symplectic automorphisms of order $2$ (``self $2$-isogenies'').

\end{theorem}

\begin{proof}
The singular $K3$-surfaces of the Ap\'ery-Fermi family defined for $k$ rational integer are 

\[Y_0,\qquad Y_2,\qquad Y_3,\qquad Y_6,\qquad Y_{10},\qquad Y_{18},\qquad Y_{102},\qquad Y_{198}. \]
This list has been computed numerically by Boyd \cite{Boy}.
Using the notation \cite{SZ}, that is writing the transcendental lattice $T_Y=
\begin{pmatrix}
a & b \\
b & c
\end{pmatrix}$ as $T_Y=[a \; b\; c]$ we get:
\[T_{Y_0}=[4 \quad 2 \quad 4] \qquad T_{Y_2}=[2 \quad 0 \quad 4]\qquad T_{Y_6}=[2 \quad 0 \quad 12].\]

They are obtained by specialisation of fibration $\#20$ for $k=0,$ $2$ and $6$. For $k=0$ the elliptic fibration has rank $0$ and singular fibers of type $I_{12}$, $I_4$, $2I_3$. For $k=2$, the transcendental lattice is already known. For $k=6$, the elliptic fibration has rank $0$ and type of singular fibers $I_{12}$, $I_3$, $2I_2$. Now using Shimada-Zhang table \cite{SZ}, we derive the previous announced transcendental lattices.

The transcendental lattices $T_{Y_3}$ and $T_{Y_{18}}$ were computed in the  paper \cite{BFFLM}. With the method used there, we can compute the transcendental lattices of $Y_{10}$, $Y_{102}$ and $Y_{198}$.
We obtain:
\[T_{Y_3}=[2\quad 1 \quad 8] \qquad T_{Y_{10}}=[6 \quad 0 \quad 12] \qquad T_{Y_{18}}=[10 \quad 0 \quad 12]\]
\[T_{Y_{102}}=[12 \quad 0 \quad 26]\qquad \qquad  [T_{Y_{198}}=[12 \quad 0 \quad 34].\]

Applying Bessi\`ere, Sarti and Veniani's Theorem, we conclude that only $Y_2$ and $Y_{10}$ may have self isogenies. By Theorem 1.2, $Y_2$ has self isogenies. We shall prove that $Y_{10}$ satisfies the same property. 

Consider the following elliptic fibration of rank $0$ of $Y_{10}$ (other interesting properties of $Y_{10}$ will be studied in a forthcoming paper):
\[y^2=x^3+x^2(9(t+5)(t+3)+(t+9)^2)-xt^3(t+5)^2\]
with singular fibers $III^*(\infty)$, $I_6(0)$, $I_4(-5)$, $I_3(-9)$, $I_2(-4)$ and $2$-torsion.
Its $2$-isogenous curve has a Weierstrass equation
\[Y^2=X^3+X^2(-20t^2-180t-432)+4X(t+4)^2(t+9)^3\]
with singular fibers  $III^*(\infty)$, $I_6(-9)$, $I_4(-4)$, $I_3(0)$, $I_2(-5)$, rank $0$ and $2$-torsion.
Hence this $2$-isogeny defines an automorphism of order $2$ of $Y_{10}$ given by $x=-\frac{X}{2}$, $y=\frac{iY}{2\sqrt{2}}$.

\end{proof}

Moreover we observe that 
\[T_{Y_2}=[2 \; 0 \; 4], \qquad T_{Y_2,\mathbb Q}=[2 \; 0 \; 1],\]
\[T_{K_2}=[4 \; 0 \; 8], \qquad T_{K_2,\mathbb Q}=[2 \; 0 \; 1], \]
 Similarly
\[ T_{Y_{10},\mathbb Q}=[6 \; 0 \; 3], \qquad  T_{K_{10},\mathbb Q}=[3 \; 0 \; 6]. \]

Hence we suspect some relations between the transcendental lattices of $K_i$ and of $S_i$ for singular $Y_i$. We give some examples of such relations in the following proposition.

\begin{proposition}
Even if the $2$-isogenies from $Y_0$, $Y_6$ are not isometries, the following rational transcendental lattices satisfy the relations
\begin{enumerate}
\item $T_{K_0,\mathbb Q}=T_{S_0,\mathbb Q}$,
\item $T_{K_6,\mathbb Q}=T_{S_6,\mathbb Q}$,
\item $K_3=S_3$.
\end{enumerate}

\end{proposition}

\begin{proof}
\begin{enumerate}
\item For $k=0$ we get two elliptic fibrations of rank $0$, namely $\#20$ and $\#8$. The fibration $\#8$-i gives a rank $0$ elliptic fibration of $K_0$ with Weierstrass equation
\[y^2=x^3+2x^2(t^3+1)+x(t-1)^2(t^2+t+1)^2,\]
type of singular fibers $D_7$, $3A_3$, $A_2$, $4$-torsion and $T_{K_0}=[8 \; 4 \; 8]$.
On the other end the fibration $\#20$-i gives a rank $0$ elliptic fibration of $S_0$
\[y^2=x(x-\frac{1}{4}(t-3I)(t+I)^3)(x-\frac{1}{4}(t+3I)(t-I)^3)\]
with type of singular fibers $3A_5$ ($\infty ,\pm I$), $3A_1$ ($0,\pm 3I$), $\mathbb Z/2 \times \mathbb Z/6$-torsion, the $3$-torsion points being $(\frac{1}{4}(t^2+1)^2,\pm \frac{1}{2}(t^2+1)^2)$. Hence, by Shimada-Zhang 's list
$T_{S_0}=[2 \; 0 \; 6]$.
Now we can easily deduce the relation
\[
\begin{pmatrix}
1/2 & 0 \\
1/2 & -1
\end{pmatrix}
\begin{pmatrix}
8 & 4 \\
4 & 8
\end{pmatrix}
\begin{pmatrix}
1/2 & 1/2 \\
0  & -1
\end{pmatrix} =
\begin{pmatrix}
2 & 0 \\
0 & 6
\end{pmatrix}.
\]

\item For $k=6$ the elliptic fibration $\#20$ has rank $0$ and $\#20-i$ gives a rank $0$ elliptic fibration of $S_6$:
\[y^2=x^3+x^2(-\frac{t^4}{2}+6t^3-21t^2+18t+\frac{3}{2})+x\frac{(t-3)^2}{16}(t^2-6t+1)^3,\]
with singular fibers $2I_6(t^2-6t+1=0)$, $I_6(\infty)$, $I_4(3)$, $2I_1(0,6)$, and $\mathbb Z/6\mathbb Z$-torsion. Using Shimada-Zhang's list \cite{SZ}, we find $T_{S_6}=[4 \; 0 \; 6]$. Since 
\[T_{K_6}=\begin{pmatrix}
           4 &  0  \\
           0 &  24
\end{pmatrix}
\underset{\mathbb Q}{\sim}
\begin{pmatrix}
1 & 0 \\
0 & 6
\end{pmatrix}\]
and
\[T_{S_6}=\begin{pmatrix}
           4 &  0  \\
           0 &  6
\end{pmatrix}
\underset{\mathbb Q}{\sim}
\begin{pmatrix}
1 & 0 \\
0 & 6
\end{pmatrix}\]

we get straightforward 
\[T_{K_6,\mathbb Q}=T_{S_6,\mathbb Q}.\]

\item Consider the elliptic fibration $\#20$ of $Y_3$ with Weierstrass equation
\[y^2=x^3+\frac{1}{4}(t^4-6t^3+15t^2-18t-3)x^2-t(t-3)x,\]
singular fibers $I_{12}(\infty)$, $2I_3(t^2-3t+1=0)$, $2I_2(0,3)$, $2I_2(t^2-3t+9=0)$, rank $1$ and $6$-torsion. The infinite section $P_3=(t,-\frac{1}{2}t(t^2-3t+3))$, of height $\frac{5}{4}$ generates the free part of the Mordell-Weil group, since $\det(T_{Y_3})=15$ by the previous theorem and by the Shioda-Tate formula 
\[\det(T_{Y_3})=\frac{5}{4} \frac{12\times 3^2 \times 2^2}{6^2}=15.\]
Its $2$-isogenous curve has Weierstrass equation

\[y^2=x^3+(-\frac{1}{2}t^4+3t^3-\frac{15}{2}t^2+9t+\frac{3}{2})x^2+\frac{1}{16}(t^2-3t+9)(t^2-3t+1)x,\]
singular fibers $3I_6(\infty, t^2-3t+1=0)$, $2I_2(t^2-3t+9=0)$, $2I_1(3,0)$, rank $1$ and $6$-torsion. The section $Q_3$ image by the $2$-isogeny of the infinite section $P_3$ is an infinite section of height $\frac{5}{2}$. Since neither $Q_3$ nor $Q_3+(0,0)$ are $2$-divisible, the section $Q_3$ generates the free part of the Mordell-Weil group. Hence by the Shioda-Tate formula, it follows
\[\det(T_{S_3})=\frac{5}{2}\frac{6^3 \times 2^2}{6^2}=60=\det(T_{K_3}).\]
We can show that $K_{3}$ and $S_{3}$ are the same surface. To prove this
property we show that a genus one fibration is indeed an elliptic fibration. 
We start with the fibration of $K_{3}$ obtained from $\#26$-i and parameter $m=\frac{y}{t\left(  x+\frac{1}{4}\left(  t-s\right)
^{2}\left(  ts-1\right)  ^{2}\right)  }.$ %(page $22-23$)
 If $k=3$ and $s=s_{3}:=\frac{3+\sqrt{5}}{2}$ we get
$E_{m}.$ Then changing $X=s_{3}^{2}x$ and $Y=s_{3}^{3}y$  it follows
\[
y^{2}-3m y x=x\left(  x-m^{2}\right)  \left(  x-\frac{1}{8}\left(  \left(
112-48\sqrt{5}\right)  m^{4}-16m^{2}+7+3\sqrt{5}\right)  \right).
\]
The next fibration is obtained with the parameter $n=\frac{x}{m^{2}}$. Now  if $w=\frac{y}{m^{2}}$ it gives the following quartic
in $w$ and $m$ %
\[
w^{2}-3 m n w+2\left(  3\sqrt{5}-7\right)  n\left(  n-1\right)  m^{4}-n\left(
n-1\right)  \left(  n-2\right)  m^{2}-\frac{1}{9}\left(  3\sqrt{5}+7\right)
n\left(  n-1\right).
\]
Notice the point $\left(  w=-\frac{1}{4}\left(  7+3\sqrt{5}\right)  n\left(
n+1\right)  \left(  n-1\right)  ,m=\frac{1}{4}\left(  2+\sqrt{5}\right)
\left(  2n-1+\sqrt{5}\right)  \right)  $ on this quartic, so it is an
elliptic fibration of $K_3$ which is $\#15$-i.

\end{enumerate}
\end{proof}

\begin{remark}
The Kummer surface $K_0$ is nothing else than the Schur quartic \cite{BSV} (section 6.3) with equation
\[x^4-xy^3=z^4-zt^3.\]
\end{remark}


\begin{thebibliography}{30}

\bibitem[Beau]{Beau}  A. Beauville, \emph{Les familles stables de courbes elliptiques sur $\mathbb{P}^1$ admettant 4 fibres singuli\`eres}, C.R.Acad.Sc. Paris, 294 (1982), 657--660.

\bibitem[BGL]{BGL} M.J. Bertin, A. Garbagnati, R. Hortsch, O. Lecacheux, M. Mase, C. Salgado, U. Whitcher, \emph{ Classifications of Elliptic Fibrations of a Singular K3 Surface}, in Women in Numbers Europe, 17--49, Research Directions in Number Theory, Association for Women in Mathematics Series, Springer, 2015. 

\bibitem[BFFLM]{BFFLM} M.J. Bertin, A. Feaver, J. Fuselier, M. Lal\`{i}n, M. Manes, \emph{Mahler measure of some singular $K3$-surfaces}, Contemporary Math. 606, CRM Proceedings, Women in Numbers 2, Research Directions in Number Theory, 149--169.


\bibitem[BL]{BL} M.J. Bertin, O. Lecacheux, \emph{ Elliptic fibrations on the modular surface associated to $\Gamma_1(8)$}, in  Arithmetic and geometry of K3 surfaces and Calabi-Yau threefolds, 153--199, Fields Inst. Commun., {\bf 67}, Springer, New York, 2013.

\bibitem[BL1]{BL1} M.J. Bertin, O. Lecacheux, \emph{Automorphisms of certain Niemeier lattices and elliptic fibrations,} 
Albanian Journal of Mathematics
Volume 11, Number 1 (2017)  13–-34. 
%ISSN: 1930-1235; (2017)

\bibitem[BSV]{BSV} S. Boissi\`ere, A. Sarti, D. Veniani, \emph{ On prime degree isogenies between $K3$ surfaces,} Rendiconti del Circolo Matematico di Palermo Series 2, Volume 66, Issue 1, (2017),  3–-18.

\bibitem[Bou]{Bou}
N. Bourbaki,  \emph{Groupes et alg\` ebres de Lie}, Chap.4, 5, 6, Masson, Paris (1981).

\bibitem[Boy]{Boy}
D. Boyd, \emph{Private communication}.

\bibitem[C]{C}P. Comparin, A. Garbagnati, \emph{Van Geemen-Sarti involutions and elliptic
fibrations on K3 surfaces double cover of $\mathbb P^2$,} J. Math. Soc. Japan
Volume 66, Number 2 (2014), 479--522. %\ ArXiv 1110.6380V2

\bibitem[Co]{Co}
J. H. Conway, N. J. A. Sloane, \emph{Sphere packings, lattices and groups}, Grundlehren der Mathematischen Wissenshaften {\text{290}}, Second Edition, Springer-Verlag, (1993).


\bibitem[D]{D} I.V. Dolgachev, \emph{ Mirror symmetry for lattice polarized $K3$ surfaces}, Journal of Mathematical Sciences, 81:3, (1996),  2599–-2630.  %arXiv:alg-geom/9502005v2, 7 Jan. 1996, 1-40.

\bibitem [DG]{DG} E. Dardanelli, B. van Geemen, \emph{Hessians and the moduli space of cubic surfaces}, in Algebraic Geometry Korean-Japan Conference in Honor of I. Dolgachev'60th Birthday 2004, Contemp. Math.  422 (2007) 17--36. AMS. 


\bibitem[E]{E} N. Elkies, \emph{Private communication}.


\bibitem[El]{El} N. Elkies, A. Kumar, \emph{$K3$ surfaces and equations for Hilbert modular surfaces,} Algebra Number Theory,
Volume 8, Number 10 (2014), 2297--2411.


\bibitem[G]{G} B. van Geemen, A. Sarti, \emph{Nikulin Involutions on $K3$ Surfaces}, Math. Z.
255 (2007),731--753.

\bibitem[K]{K} J. Keum, \emph{A note on elliptic $K3$ surfaces}, Transactions of the American Mathematical Society,
Vol. 352, No. 5 (2000), 2077--2086.

\bibitem[Ko]{Ko}K. Koike,  \emph{Elliptic $K3$ surfaces admitting a Shioda-Inose structure}, Commentarii Mathematici Universitatis Sancti Pauli Vol 61 N° 1 (2012), 77--86. 
%preprint (2011), arXiv:1104.1470.

\bibitem[Ku]{Ku}
M. Kuwata, T. Shioda,  \emph{ Elliptic parameters and defining equations for elliptic fibrations on a Kummer surface},
 Advanced Studies in Pure Mathematics 50, 2008, Algebraic Geometry in East Asia - Hanoi 2005, 177-215.
 
\bibitem[Ku1]{Ku1}
M. Kuwata, "Maple Library 'Elliptic Surface Calculator'",

http://c-faculty.chuo-u.ac.jp/$\sim$ kuwata/ESC.php.



\bibitem[M]{M} D. Morrison, \emph{ On $K3$ surfaces with large Picard number}, Invent. Math. 75 (1984), 105--121.

\bibitem[Na] {Na} N. Narumiya, 
H. Shiga, \emph{ The mirror map for a family of $K3$ surfaces induced from
the simplest $3$-dimensional reflexive polytope,} Proceedings on Moonshine and
related topics (Montr\'{e}al, QC, 1999),  CRM Proc. Lecture Notes, 30
(2001), 139--161, Amer. Math. Soc.

\bibitem[Nik]{Nik}
V. V. Nikulin, \emph{Integral symmetric bilinear forms and some of their applications}, Math. USSR Izv. {\bf{14}}, No. 1 (1980), 103--167.


\bibitem[N1]{N1} V. V. Nikulin,\emph{ Finite automorphism groups of Kahler surfaces of type K3}, Proc. Moscow Math. Soc. 38 (1979), 75--137.


\bibitem[Nis]{Nis}
 K.-I. Nishiyama, \emph{The Jacobian fibrations on some $K3$ surfaces and their Mordell-Weil groups}, Japan. J. Math. {\textbf{22}} (1996), 293--347.


\bibitem[PA]{PA}
The PARI Group, Bordeaux GP/PARI version 2.7.3, 2015, http://pari.math.u-bordeaux.fr.

\bibitem[PS]{PS} C. Peters, J. Stienstra,  \emph{A pencil of $K3$-surfaces related to Ap\'ery's recurrence for $\zeta (3)$ and Fermi surfaces for potential zero}, Arithmetic of Complex Manifolds (Erlangen, 1988) (W.-P. Barth \& H. Lange, eds.), Lecture Notes in Math., vol. 1399, Springer, Berlin (1989), 110--127.

\bibitem[Sc]{Sc} M. Sch\"{u}tt, \emph{Sandwich theorems for Shioda-Inose structures}, Izvestiya Mat. 77 (2013), 211--222.



\bibitem[Sc-Shio]{Sc-Shio}
M. Sch\"{u}tt, T. Shioda,  \emph{Elliptic surfaces}, Algebraic geometry in East Asia - Seoul 2008, Adv. Stud. Pure Math. {\textbf{60}} (2010), 51--160.






\bibitem[Shio]{Shio}
T. Shioda, \emph{Kummer sandwich theorem of certain elliptic $K3$ surfaces}, Proc. Japan Acad. \textbf{82}, Ser. A (2006), 137--140.

\bibitem[Shio1]{Shio1}
T. Shioda  \emph{ On elliptic modular surfaces} J. Math. Soc. Japan 24 
(1972), no 1,  20--59.



\bibitem[SI]{SI}
T. Shioda, H. Inose, \emph{On singular $K3$ surfaces}, in: Baily, W. L. Jr., Shioda, T. (eds), Complex analysis and algebraic geometry, Iwanami Shoten, Tokyo (1977), 119--136.

 \bibitem[SM]{SM}
T. Shioda, N. Mitani, \emph{Singular abelian surfaces and binary quadratic forms}, in: Classification of algebraic varieties and compact complex manifolds, Lect. Notes in Math. 412 (1974).

\bibitem[SZ]{SZ} I. Shimada, D. Q. Zhang, \emph{
Classification of extremal elliptic $K3$ surfaces and fundamental groups of open $K3$ surfaces}, Nagoya Math. J. {\bf{161}} (2001), 23--54.



\bibitem[Si]{Si}
J. Silverman, The arithmetic of Elliptic Curves, Graduate texts in mathematics Springer 106 (1986).


%\bibitem[V]{V} H. Verrill, 


















\end{thebibliography}
\end{document}